\documentclass[12pt]{article}
\usepackage[utf8]{inputenc}
\usepackage{graphicx}
\usepackage{caption}
\usepackage{latexsym}
\usepackage[english]{babel}

\usepackage{amscd,amsfonts,amsmath,amssymb,amsthm}
\usepackage[mathscr]{eucal}

\usepackage[usenames]{color}

\input{xy}
\xyoption{all}

\usepackage{tikz-cd}
\usepackage{tikz}
\usetikzlibrary{decorations.pathreplacing}

\usepackage{tensor} 
\usepackage{mathdots} 

\usepackage[toc,page]{appendix} 
\usepackage{hyperref}
\hypersetup{
    colorlinks=true,
    linkcolor=blue,
    filecolor=magenta,
    urlcolor=magenta,
    citecolor=blue
}

\usepackage[draft]{changes}   
\usepackage{framed}

\newcommand{\adj}[1]{{#1}^*}
\renewcommand{\deg}{\text{degree}}
\newcommand{\gt}{\theta}

\newcommand{\R}{\mathbb{R}}
\newcommand{\C}{\mathbb{C}}

\theoremstyle{plain}
\newtheorem{theorem}{Theorem}

\newtheorem{lemma}{Lemma}
\newtheorem{cor}{Corollary}
\newtheorem{proposition}{Proposition}

\theoremstyle{definition}

\theoremstyle{remark}
\newtheorem{remark}{Remark}

\newcounter{dimension}

\title{
On the Irreducibility of the Differential Operators Associated to Random Walks in the Standard Euclidean Lattice
}
\author{Dorin Dumitra\c{s}cu\footnote{Corresponding author: \texttt{ddumitrascu@adrian.edu}} and Liviu Suciu}
\date{\today}

\begin{document}

\maketitle

\begin{abstract}
\noindent
For a positive integer $d\geq 1$, we consider the sequences $(A_{n}^{(d)})_n$ and $(x_{n}^{(d)})_n$ given by
$$
A_{n}^{(d)} =\sum_{n_1+\dots+n_d=n} \frac{(2n)!}{(n_1!)^2 (n_2!)^2 \dots (n_d!)^2}
\quad \text{ and } \quad
x_{n}^{(d)} = \frac{A_{n}^{(d)}}{\binom{2n}{n}}.
$$
They have rich combinatorial interpretations, but we focus on the analytical properties of their generating functions $A_d$ and $F_d$.
We use a modified Borel transform, and algebraic and combinatorial considerations to prove that $F_d$ is annihilated by an irreducible Fuchsian differentiable operator $L_{d-1,F}$ of order $d-1$. We determine the structure of $F_d$ as a global analytic function (analytic continuations from the original disk of definition, branches, finite singularities, and the structure of $F_d$ near the finite singularities). Additionally, we show that the sequence $(x_n^{(d)})_n$ satisfies a minimal recurrence of width $r=\lfloor (d+1)/2 \rfloor$ with polynomial coefficients
$$
Q_r(n+r)\,x_{n+r}+\cdots + Q_0(n)\,x_n=0, \; n \ge 0.
$$
These polynomials are shown to have very specific symmetries and we compute explicitly $Q_0$, $Q_1$, and $Q_r$. Similar results about the functions $A_d$ are obtained.
\end{abstract}

\vspace{10pt}
\noindent{\it 2020 Mathematics Subject Classification}:
Primary
05A15; 
Secondary
05A10, 
13N10, 
33F99, 
34M35, 
68W30. 

\vspace{7pt}
\noindent{\it Keywords}: random walk, generating function,  differential operator, $P$-recurrence, singularity analysis, period integral, formal self-adjointness, Hadamard convolution, Borel transform, symmetric powers, Sylvester-Kac tri-diagonal matrix.

\newpage
\tableofcontents

\section{Introduction}\label{sec-1-introduction}
For a positive integer $d\geq 1$, we consider the sequences $(A_{n}^{(d)})_n$ and $(x_{n}^{(d)})_n$ given, respectively, by
$$
A_{n}^{(d)} =\sum_{n_1+\dots+n_d=n} \frac{(2n)!}{(n_1!)^2 (n_2!)^2 \dots (n_d!)^2}
\quad \text{ and } \quad
x_{n}^{(d)} = \frac{A_{n}^{(d)}}{\binom{2n}{n}} = \sum_{n_1+\dots+n_d=n} \frac{(n!)^2}{(n_1!\, n_2! \dots n_d!)^2}.
$$
They have rich combinatorial interpretations (\cite{Domb, GP93, RS09, DS1}). For example, $A_{n}^{(d)}$ counts the number of closed random walks of length $2n$ that start at the origin of the standard lattice of $\R^d$,\footnote{In \cite{DS1} we used the notation $A_{2n}^{(d)}$ to make explicit the length of the walks, but here, for notational convenience, we are using $A_n^{(d)}$.} and the sequence $(x_{n}^{(d)})_n$ appears in the study of staircase polygons and  abelian squares. The sequence $(x_{n}^{(d)})_n$ also appears naturally if one considers the recurrence relation satisfied by the $(A_{n}^{(d)})_n$ sequence (\cite[Corollary~2.1]{DS1}):

\begin{equation}\label{eqn-A-recurrence}
A_{n}^{(d+1)} = \binom{2n}{n} \cdot \sum_{k=0}^{n} \binom{n}{k}^2 \,
        \frac{A_{k}^{(d)}}{\binom{2k}{k}}.
\end{equation}

\noindent
The two sequences define the analytic generating functions
$$
A_d(z) =\sum_{n \ge 0} A_n^{(d)} \, z^n
\quad \text{ and } \quad
F_d(z)=\sum_{n \ge 0} x_n^{(d)} \, z^n.
$$
Note that, for all $n\geq 0$,  $x_n^{(1)}=1$ and $x_n^{(2)}=A_n^{(1)}=\binom{2n}{n}$. This means that $F_1(z)=1/(1-z)$ and $F_2(z)=A_1(z)=1/\sqrt{1-4z}$.

\vspace{.4cm}
There are at least four ways of thinking of the sequences $(x_n^{(d)})_n$ and the associated power series $F_d$: inductive through the Legendre process at the recurrence level, inductive through Hadamard convolution at the analytic level, through expressing $F_d$ as a period integral and discussing the differential operator that annihilates it, and through a modified Borel transform of $F_d$.

We detail each of these four ways of thinking. The recurrence is (see equation (\ref{eqn-A-recurrence}) above):

$$
x_n^{(d+1)} = \sum_{k=0}^n {n \choose k}^2 x_k^{(d)}, \text{ for } n \ge 1, d \ge 1,
$$
with $x_n^{(1)}=1$. The Hadamard convolution is given by
$$
F_{d+1}(w) = \frac{1}{2\pi i}\int_{|t|=1/s_d}
    F_d(t)P(t,w)^{-1/2}dt,
$$
for $P(t,w)=w^2-wt(2w+2)+t^2(w-1)^2$, $s_d$ a small enough number, and $F_{1}(w) = \frac{1}{1-w}$. See \cite[Subsection~3.4]{DS1}.
The period integral is given by
\[
F_d(z)=\frac{1}{(2\pi i)^d}\oint \cdots \oint
\frac{1}{1-zg_x(\mathbf{u})}\,\frac{du_1}{u_1}\cdots\frac{du_d}{u_d}
=
\sum_{k=0}^{\infty} z^{k}
\sum_{k_1+\cdots+k_d=k}
\frac{(k!)^2}{(k_1!\cdots k_d!)^2},
\]
where $g_x$ is the Laurent polynomial $g_x(\mathbf{u})=(u_1+\cdots + u_d)(1/u_1+\cdots +1/u_d)$. Since $g_x(\lambda \mathbf{u})=g_x(\mathbf{u})$, for $\lambda \neq 0$, the Laurent polynomial has scalar symmetry in addition to the obvious discrete symmetries coming from permuting the variables $u_1,\cdots, u_d$. Hence it is expected that the corresponding differential operator obtained from this period representation would be of order $(d-1)$ and we will prove that this is indeed the case. The modified Borel transform is given by
$$
Y_d(t)=\sum_{n \ge 0} \frac{x_n^{(d)}}{(n!)^2} \, t^n
 = \sum_{n \ge 0} y_n^{(d)} \, t^n, \text{ with }
 y_n^{(d)}=\frac{x_n^{(d)}}{(n!)^2},
$$
and satisfies the property that $Y_d = Y_1^d$.

It follows that one can think of $(A_n^{(d)})_n$ in the corresponding four ways. The recurrence is given by equation (\ref{eqn-A-recurrence}). The Hadamard convolution is given by $A_d(z)=F_d(z)*A_1(z)=F_d(z)*F_2(z)$. The period integral is given by
\[
A_d(z)=\frac{1}{(2\pi i)^d}\oint \cdots \oint
\frac{1}{1-zg_A(\mathbf{u})}\,\frac{du_1}{u_1}\cdots\frac{du_d}{u_d}
=
\sum_{k=0}^{\infty} z^{k}
\sum_{k_1+\cdots+k_d=2k}
\frac{(2k)!}{(k_1!\cdots k_d!)^2},
\]
where $g_A$ is the Laurent polynomial $g_A(\mathbf{u})=u_1+1/u_1 +\cdots + u_d+1/u_d$. This polynomial $g_A$ has discrete symmetries generated by permuting the variables $u_1,\cdots, u_d$ and by the transformations $u_j\mapsto u_j^{\pm 1}$. Beukers and Vlasenko \cite{BV25} showed that, as expected, the corresponding differential operator following from this period representation is of order $d$. The modified Borel transform is given by
$$
Y_{d}(t)=\sum_{n \ge 0} \frac{A_n^{(d)}}{(2n)!} \, t^n
 = \sum_{n \ge 0} y_n^{(d)} \, t^n, \text{ since }
 y_n^{(d)}=\frac{x_n^{(d)}}{(n!)^2}=\frac{A_n^{(d)}}{(2n)!}.
$$

\vspace{.4cm}
In this paper we use a modified Borel transform, and algebraic and combinatorial considerations to prove the $F_d$ is annihilated by an irreducible, hence minimal for $F_d$, Fuchsian differentiable operator $L_{d-1,F}$ of order $d-1$. We also prove that this operator has many interesting properties (see Theorem~\ref{T-LF-properties}). We also determine the structure of $F_d$ as a global analytic function (analytic continuations from the original disk of definition, branches, finite singularities, and the structure of $F_d$ near the finite singularities). Additionally, we show that the sequence $(x_n^{(d)})_n$ satisfies a minimal recurrence of width $r=\lfloor (d+1)/2 \rfloor$, so with $r+1$ polynomial coefficients $Q_0, \cdots, Q_r$. These polynomials are shown to have very specific symmetries (even or odd with respect to a center of symmetry) and we compute explicitly  $Q_0$, $Q_1$, and $Q_r$, in any dimension $d$.

In Theorem~\ref{T-LA-properties}, using the close relationship at all levels (recurrences, generating functions, ODEs, analytic properties) between $(x_n^{(d)})_n$ and $(A_n^{(d)})_n$, we prove the corresponding results for the sequence $(A_n^{(d)})_n$ and its generating function $A_d$. In particular, we show that $A_d$ is annihilated by an irreducible Fuchsian differentiable operator $L_{d,A}$ of order $d$, confirming the conjecture from \cite{BV25}. It also follows that the sequence $(A_n^{(d)})_n$ satisfies a minimal recurrence of width $r=\lfloor (d+1)/2 \rfloor$.

\vspace{.4cm}
The structure of the paper is as follows. In Section~\ref{sec-2-generalities} we recall some general properties of linear differential operators with rational functions coefficients and present some basic results that are used later in the paper. In Section~\ref{sec-3-framework} we construct our framework that connects the data associated with a linear differential operator $L$. This allows us to move between  the standard form (where the basic differential operator is $D=d/dz$), the $\gt$-form (where the basic differential operator is $\gt = z \, d/dz$), and the $P$-recurrence of the coefficients when there is a power series annihilated by $L$. In Section~\ref{sec-4-SA-ASA} we discuss the notion of formal-self-adjointness for such an operator $L$ written in (canonical) polynomial form, how it translates to the $\gt$-form of $L$, and how it is expressed in terms of the polynomial coefficients of the $P$-recurrence for a power series annihilated by $L$. After presenting in Section~\ref{sec-Main} our main results mentioned above, we provide the proofs in Section~\ref{sec-proofs}. These proofs use mostly algebraic and combinatorial methods, including classical facts about the Sylvester-Kac tri-diagonal matrix. For the structure of $F_d$ and $A_d$ as global (multi-valued) analytic functions we expand on the Hadamard convolution techniques from our original paper \cite{DS1}. We end the paper with Section~\ref{sec-Appendix} which serves as an Appendix. We present the framework data (recurrences, and the $L$ operators in both standard and $\gt$-forms) for dimensions $d=3$~to~$8$.

\section{Generalities about ODEs}\label{sec-2-generalities}

We work in $\mathbb{C}(z)\langle D\rangle$, the space of linear differential operators with rational functions coefficients, where as usual $D=d/dz$ and notice that $Dz=zD+1$.
Since the space of coefficients $\mathbb{C}(z)$ does not commute with $D$, it is important to specify a ``canonical" form for $L$. Hence we say that a differential operator $L$ is in {\it canonical polynomial form (CPF)} if

$$
L=\sum_{k=0}^{q} P_k(z) D^k, \text{ with } P_k(z)\in \mathbb{C}[z] \text{ polynomials without (non-constant) common factors.}
$$

\noindent
As usual, $q$ is the order of $L$ and the roots of $P_q(z)$ are the finite singularities of $L$. Any other finite point is called ordinary. By general theory, the solution space at any finite singularity $z=a$ is of dimension $q$ and we call $z=a$ {\it Fuchsian singularity} if all solutions are locally bounded by a fixed power of $z-a$. This is also called a {\it regular singular point} and the given condition is equivalent to the indicial equation having degree $q$. Assuming without loss of generality that $a=0$, the Fuchsian condition is that $P_k(z)/P_q(z)$ has a pole of order at most $q-k$. If $c_k$ is the coefficient of this pole, the indicial equation is

$$
(\lambda)_q +c_{q-1}(\lambda)_{q-1}+\dots + c_1\lambda + c_0 = 0,
$$
where  $(\lambda)_k = \lambda (\lambda-1)\cdots (\lambda-k+1)$ are the falling factorials. The roots are called {\it the symbols} or {\it the local exponents} at $a$.
If the solution space at a finite singularity consists only of analytic functions, we call that an apparent singularity.
At $z=\infty$, we use the change of variable $w=1/z$ and we study the transformed $L_w$ operator at $w=0$. As is well known (see \cite{Ince}), the Fuchsianity at $\infty$ is equivalent to

$$
\text{degree}(P_q)\geq  \text{degree}(P_k)+q-k.
$$

\vspace{.2cm}
\begin{lemma}\label{L-general}
Let $L$ be a Fuchsian differential operator (meaning that all finite singularities and $\infty$ are Fuchsian). We assume that the order $q\geq 1$, $L$ is in CPF, and $L$ has $N$ singularities (the roots of $P_q$ together with $\infty$, if $0$ is a singularity of the transformed $L_w$ operator). Then:

\begin{enumerate}
\item The Frobenius relation holds:
$$
\text{sum on all symbols}= \frac{q(q-1)}{2} (N-2).
$$

\item If $N \geq 2$ and $z=0, a_1, \dots, a_{N-2}, \infty$ are the singularities of $L$
such that all symbols at all of them are non-negative numbers and the sum of the symbols at each of $a_1, \dots, a_{N-2}$ is at least $\displaystyle{ \frac{q(q-1)}{2}}$, then the symbols of $L$ at $0$ and $\infty$ are all zero. (Here we allow $N=2$, where the condition on the $a_k$'s is trivially satisfied, since there are none.)

\end{enumerate}
\end{lemma}

\noindent
{\it Remarks.}
If $P_q$ is constant, then the Fuchsianity at $\infty$ implies that all $P_k$ are zero for all $k<q$ so $L=cD^q$, and we easily see that $\infty$ is a singularity precisely if $q \geq 2$, so $N=0$ when $L=cD$, $c \neq 0$, and $N=1$ if $L=cD^q$, $q \geq 2$, $c \neq 0$.

If $P_q$ has only one root, we use a linear transformation to put the singularity at $z=0$. So, $P_q(z)=c_q \, z^m$, $m\geq 1$. The Fuchsianity at $z=0$ and $\infty$ implies that $P_{q-1}(z)=c_{q-1} \, z^{m-1}$, with $c_{q-1}\neq 0$. Continuing inductively, the CPF shows that $m\leq q$ and
$$
L=c_q \, z^m D^k + c_{q-1} \, z^{m-1} D^{q-1} +\cdots +
c_{q-m}\, D^{q-m}, \text{ with all } c_k\neq 0.
$$
The transform $w=1/z$ shows that $\infty$ is a singularity unless the transformed operator (after simplifying the highest common power of $w$) becomes $L_{\infty}=cD_w^q, q \geq 2$.
So in general $N \geq 2$, except in the special cases above ($L=cD^q$ or a transform of it under $w-a=\frac{1}{z}$)

Also, at a non-singular point $a$, the symbol is $0,1, \cdots, q-1$, so the Frobenius relation is consistent, in the sense that artificially considering several ordinary points as ``singular," still preserves it.

\vspace{.2cm}
\begin{proof}
Part 1 is a classical well known result \cite{Ince}. Part 2 follows from Part~1 and the non-negativity of all symbols.
\end{proof}

\vspace{.2cm}
Assuming $L$ is Fuchsian at $z=0$, we say that $L$ {\it is of MUM-type (maximal unipotent monodromy) at z=0} if the symbols at $z=0$ are all zero. This is equivalent to the structure of the local solution space being a full Jordan logarithmic block under local monodromy.

\vspace{.5cm}
\begin{lemma}\label{L-symbols}
Let $L \in \mathbb{C}(z)\langle D\rangle$ be in canonical polynomial form with a Fuchsian singularity at $z=0$ (so $P_q(0)=0$), then if $P_q(z)=z^m \widetilde{P}_q(z)$, with $\widetilde{P}_q(0)\neq 0$, the following hold.

\begin{enumerate}
\item $1\leq m \leq q$.
\item If $z=0$ is MUM, then $m=q-1$.
\item If $z=0$ is a simple singularity ($m=1$) then the symbol is $\{0, 1, 2, \cdots, q-2, \alpha\}$, where $\alpha$ could be any complex number.
\end{enumerate}

\end{lemma}
\begin{proof}
    By assumption $m\geq 1$, and by Fuchsianity $z^{m-q+k}$ divides $P_k$ as long as $m-q+k\geq 1$. If $m\geq q+1$ then $m-q+k\geq 1$ for all $0\leq k \leq q$, so $z$ divides all the polyunomials $P_0$, $P_1$, etc. which contradicts our assumption that $L$ is in CPF. For the second part, the condition MUM at $z=0$ means that the indicial equation is $\lambda^q=0$. It is well known that, for $q\geq 1$,
$$
\lambda^q = \sum_{k=0}^{q}S(q,k) (\lambda)_k,
$$
where $S(q,k)$ are the Stirling numbers of second kind, with $S(q,0)=0$ and $S(q,k)>0$, for $1\leq k\leq q$. If $c_k$ are the coefficients of the indicial equation at $z=0$, then $c_k=S(q,k)$, so $c_0=0$ and $c_k>0$, for $1\leq k\leq q$. This implies that
${P_k(z)}/{z^m}$ has a pole of order of exactly $q-k$, or $q\geq 1$, which means that $m\geq q-1$. If $m=q$, then $z$ divides all $P_k$, for $1\leq k \leq q$, while now $c_0=P_0(0)=0$ implies that $z$ divides $P_0$ as well. This would again contradict CPF. The last part follows because the indicial equation is then $(\lambda)_q + c_{q-1} (\lambda)_{q-1}=0$, which has the roots $0, 1, \cdots, q-2$, and $q-1-c_{q-1}$.
\end{proof}

It is well known \cite[Section4.4]{Kauers} that that $\mathbb{C}(z)\langle D\rangle$ has right factorization (which is far from unique), so if $L$ is reducible in the sense that its differential module is reducible then $L=AB$, with $A$ and $B$ also in $\mathbb{C}(z)\langle D\rangle$ and $\text{order}(A) + \text{order}(B) = \text{order}(L)$. Note that even if $L$ is in CPF then $A$ and $B$ may not have polynomial coefficients. For example, $D^2=(D+\tfrac{1}{z+\alpha})(D-\tfrac{1}{z+\alpha})$, for any $\alpha \in \C$.

\vspace{.5cm}
\begin{lemma}\label{L-factorization}
Let $L \in \mathbb{C}(z)\langle D\rangle$ be a linear differential operator and suppose $L = A B$ is a factorization.

\begin{enumerate}

\item Let $z=a$ be a singular point of $L$. Then
$
\mathrm{Sol}_a(B) \subseteq \mathrm{Sol}_a(L).
$
In particular, the local solution space of $B$ is a subspace of that of $L$, so if $a$ is a Fuchsian singularity of $L$ then it is a Fuchsian singularity of $B$. Moreover not only are the local exponents of $B$ contained in those of $L$, but the full local solution structure (including logarithmic blocks) is inherited as a substructure.

\item If $L$ is Fuchsian, then $A$ and $B$ are Fuchsian.

\item If $L$ is of MUM-type at $z=0$, then $B$ is also of MUM-type at $0$, and $B$ annihilates the unique analytic solution $F(z)=1+x_1 z+ \cdots$ of $L$.

\item Let $L$ be of MUM-type at $z=0$ and $F$ the unique solution satisfying $F(0)=1$. The operator $L$ is irreducible if and only if it is minimal for $F$, i.e. there is no nontrivial operator of lower order annihilating $F$.

\item If $B$ has other finite singularities than $L$, then these are apparent, hence at each such point the local exponents are strictly increasing nonnegative integers.

\end{enumerate}
\end{lemma}

\begin{proof}
Part~1 is obvious; Part~2 for B follows from Part~1, and for A one can use the adjoints. (See Section~\ref{sec-4-SA-ASA}.) For Part~3, using Part~1, the solution space for $B$ must be a Jordan logarithmic block, so the symbol is $\{0,0, \cdots, 0\}$, so MUM. If $F_B$ is the unique power series at $z=0$, with $B(F_B)=0$ and $F_B(0)=1$, then $L(F_B)=0$, which means that $F_B$ must be the unique normalized analytic solution of $L$ a $z=0$. Part~4 follows from Part~3, since irreducibility implies minimality, and in the other direction, if $L$ were reducible, Part~3 shows it would not be minimal. For Part~5, since the solution space of $B$ at $a$ is included in the solution space of $L$ at $a$, if follows that, at any ordinary point of $L$, $B$ must have only analytic solutions. Hence any such points are apparent singularities of $B$ and then the general structure theory shows that the local exponents must be distinct non-negative integers.
\end{proof}

\noindent
\begin{remark} \label{R:MUM-lemma} 
In Part~4 of this lemma, the MUM condition is essential because, in general, if $L$ is minimal for some analytic function $F$ it does not follow that $L$ is irreducible. For example, $F(z)=-\log(1-z)=\sum_{n\geq 1} z^n/n$ is annihilated by $L =(1-z) D^2 - D$, which is reducible of order 2 and it is easily seen that $F$ cannot be annihilated by a order 1 operator. So $L$ is minimal for $F$ but reducible.
\end{remark}

\section{Framework dictionary}\label{sec-3-framework}

We assume that all the finite singularities are Fuchsian, and we will explicitly specify when we require this condition at $\infty$.

Let $L$ be a differential operator of order $q$ in CPF with a Fuchsian singularity at a finite point $z=a$, which can be assumed to be 0 by a linear change of variable. By abusing the notation, we allow $a=0$ to also be an ordinary point, which will be used for the case $q=1$. Recall from Section~\ref{sec-2-generalities} that $L=\sum_{k=0}^{q} P_k(z) D^k$ and we set $P_q(z)=z^m \widetilde{P}_q(z)$, $\widetilde{P}_q(0)\neq 0$, and $0\leq m \leq q$ is the order of the singularity.

We introduce the based homogeneous differential operator $\theta = zD$, which satisfies $\theta z^k = k z^k$, $\theta z=z(\theta +1)$, and $\theta_w=-\theta_z$, where $w=1/z$. We notice that

$$
z^k D^k = (\theta)_k = \theta(\theta-1)\cdots (\theta -k+1).
$$
Hence
\begin{equation}
L_{\theta}:=z^{q-m}L= \sum_{k=0}^{q} \widetilde{R}_k(z) \, \theta^k,
\text{ where }
\widetilde{R}_k(z)=\sum_{n=k}^{q} c_{k,n} z^{q-n-m} P_n(z),
\end{equation}
with $c_{k,n}$ non-zero integers. We have $\widetilde{R}_q(z)=\widetilde{P}_q(z)$. The Fuchsian condition means that all $z^{q-n-m} P_n(z)$ are polynomials, so $L_\theta$ has polynomial coefficients and $q-m$ is the lowest power $n$ such that $z^n L$ has polynomial coefficients. It also follows that all $\widetilde{R}_k$'s have no non-constant factors, because $\widetilde{R}_q$ does not divide by $z$, while, for $\alpha \neq 0$, if $z-\alpha$ divided all $\widetilde{R}_k$, it would follow inductively that $z-\alpha$ divides all $P_k$'s, contradicting that $L$ is in CPF. We call $L_\theta$ the {\it canonical $\theta$-polynomial form} of $L$. Conversely, if we are given an operator $L_\theta$ in canonical $\theta$-polynomial form, the Fuchsian condition means that $\widetilde{R}_q(0)\neq 0$, the indicial equation is $\sum_{k=0}^q \widetilde{R}_k(0) \lambda^k = 0$, and is an unique $0\leq m \leq q$, the order of singularity, for which $z^{m-q} \, L_\theta$, transformed in an ODE in standard form, is in CPF.

We denote by $r$ the highest degree of the polynomials $\widetilde{R}_k$ and we call it the {\it width of $L_\theta$} (and of $L$, by abuse of notation). By identifying coefficients, we can write
\begin{equation}\label{eqn-Rs}
L_\theta =\sum_{k=0}^r z^k R_k(\theta).
\end{equation}
The indicial equation is $R_0(\lambda)=0$, hence the MUM condition is $R_0(\theta)=\theta^q$.

For further reference, we note that, if we transform $\infty$ to $0$ with $w=1/z$ , the corresponding form of $L_\theta$ is
$$
L_{\theta,w}=\sum_{k=0}^q \widetilde{R}_k^t(w) (-\theta_w)^k,
\text{ where } \widetilde{R}_k^t(w) = w^r \widetilde{R}_k(1/w).
$$
Hence the Fuchsian condition at $\infty$ of $L$ and $L_\theta$ ($\widetilde{R}_q^t(0)\neq 0$) is equivalent with $r=\text{degree}(\widetilde{R}_q)=\text{degree}(\widetilde{P}_q)$.

From now on, we assume that $L$ has at least one non-trivial analytic solution $F(z)=\sum_{n=0}^\infty x_n z^n$ at $z=0$. Then the equation $L_\theta F=0$ reads
\begin{equation}
\sum_{n=0}^\infty \sum_{k=0}^r  R_k(n) \, x_n \, z^{n+k} = 0.
\end{equation}
In particular, for $n\geq 0$, we get the recurrence
\begin{equation}
\sum_{k=0}^r  R_{r-k}(n+k) \, x_{n+k}= 0,
\end{equation}
with the initial conditions
\begin{equation}\label{rec-R-ic}
\sum_{k=0}^m  R_{m-k}(k) \, x_{k}= 0,
\text{ for } 0\leq m\leq r-1.
\end{equation}

We denote by $Q_k(t)=R_{r-k}(t)$ so the recurrence becomes
\begin{equation}\label{rec0}
\sum_{k=0}^r  Q_{k}(n+k) \, x_{n+k}= 0,
\end{equation}
with the initial conditions
\begin{equation}\label{rec0-ic}
\sum_{k=0}^m  Q_{r-m+k}(k) \, x_{k}= 0,
\text{ for } 0\leq m\leq r-1.
\end{equation}
We note that the $\text{degree}(R_0)=\text{degree}(Q_r)=q$, $\text{degree}(Q_k)\leq q$, for $0\leq k \leq r-1$, and $Q_r(\lambda)=0$ is the indicial equation at 0. The Fuchsianity at $\infty$ is equivalent to $\text{degree}(Q_0)=q$ and the indicial equation at $\infty$ is $Q_0(-\lambda)=0$.

The condition that the recurrence admits a non-trivial solution becomes $Q_0$ has at least one non-negative integer root.

We note that if $z=a$ is a finite singularity of $L$ of order $m_a$ then the width at $a$ is $r_a=\text{degree}(P_q)-m_a$. In particular, if $L$ is Fuchsian at $\infty$, then $\text{degree}(P_q)\geq q$, so $r_a\geq q-m_a$.

Conversely, if we start with a recurrence of type (\ref{rec0}) with $\deg(Q_r)=q\geq \deg(Q_k)$, for all $k$, and such that $Q_r$ has a non-negative integral root, we obtain a differential operator $L$, Fuchsian at 0, which has at least one non-trivial analytic solution $F(z)=\sum_n x_n z^n$ for which $(x_n)_n$ satisfies both (\ref{rec0}) and (\ref{rec0-ic}). $L$ is given by:
\begin{equation}
L_\theta =\sum_{k=0}^r z^{r-k} Q_k(\theta).
\end{equation}

Given any of the data $(P,\tilde{R},R,Q)$ we can uniquely obtain the others, with the caveat that $Q_r$ has a non-negative integral root, so we have a non-trivial solution for (\ref{rec0}) and (\ref{rec0-ic}).

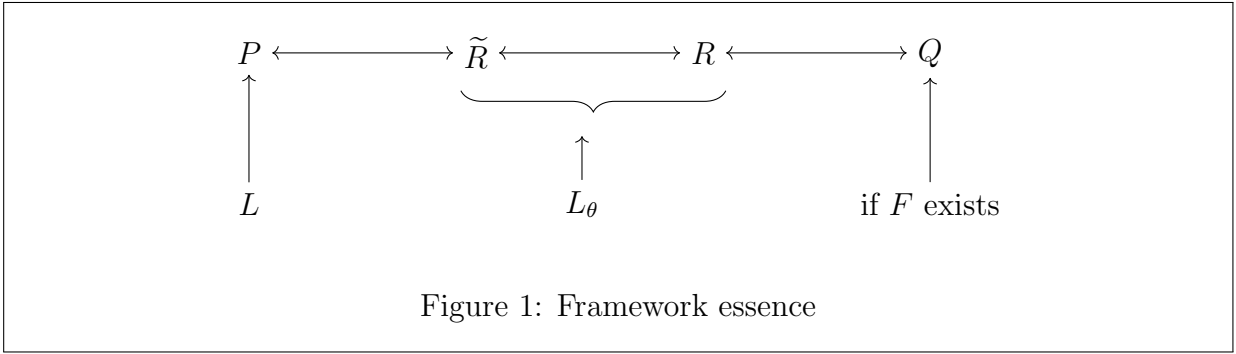
\vspace{.2cm}
\begin{figure}[h]
%

\begin{framed}

\begin{center}

\begin{tikzpicture}
\node (P) at (0,0) {$P$};
\node (Rt) at (3,0) {$\widetilde{R}$};
\node (R) at (6,0) {$R$};
\node (Q) at (9,0) {$Q$};

\draw[<->] (P) -- (Rt);
\draw[<->] (Rt) -- (R);
\draw[<->] (R) -- (Q);

\node (L) at (0,-2) {$L$};
\node (Lt) at (4.4,-2) {$L_\theta$};
\node (F) at (9,-2) {\text{if $F$ exists}};

\draw[->] (L) -- (P);
\draw[->] (Lt) -- (4.4,-1.1);
\draw[->] (F) -- (Q);

\draw[decorate,decoration={brace,mirror,amplitude=8pt}] (2.8,-0.5) -- (6.3,-0.5);
\end{tikzpicture}

\end{center}
\caption{Framework essence}

\end{framed}
\end{figure}

\noindent
\begin{remark}\label{R:spc} 
In our framework, the finite non-zero singularities are the solutions of $\widetilde{R}_q(z)=0$, and $\widetilde{R}_q(z)=\sum_{k=0}^r [\theta^q]R_k \, z^k$, where $[\theta^q]R_k$ is the $\theta^q$-coefficient of $R_k$.
\end{remark}

\vspace{.2cm}
\begin{lemma}\label{L-structure}
Let $L \in \mathbb{C}(z)\langle D\rangle$ be of order $q\geq 3$ in canonical polynomial form with a simple singularity at $z=a$ with symbol $(0, 1, 2\dots, q-2,\alpha)$. Then if $\alpha$ is not an integer greater or equal to $q-1$, then $L$ has precisely $q-1$ analytic solutions at $z=a$.
\end{lemma}

\noindent
\begin{remark}
The condition is necessary because at an ordinary point with $q$ analytic solutions $\alpha=q-1$ and one can construct similar examples for any $\alpha=q-1+n$. If $\alpha$ is not integral then we have a solution of the type $z^\alpha \cdot \text{analytic}$, while if $\alpha\in\{0,1,\cdots, q-2\}$ we have logarithmic solutions.
\end{remark}

\begin{proof}
By a linear change of variable, we assume $a=0$. Let $L=\sum_{k=0}^q P_k(z) D^k$, $P_q=z\, \widetilde{P}_q$, and write $P_k(z)=\sum_{m\geq 0} p_{k,m}z^m$. Since $z^k D^k=(\theta)_k$, the canonical $\theta$-form is
$$
L_\theta=z^{q-1}L=\sum_{k=0}^q z^{q-1-k} P_k(z) \, (\theta)_k
= \sum_{k=0}^q \sum_{m\geq 0} p_{k,m} z^{m+q-1-k} \, (\theta)_k.
$$
Hence the coefficient $R_j(\theta)$ of $z^j$ is
$$
R_j(\theta) = \sum_{m+q-1-k=j} p_{k,m} (\theta)_k.
$$
We also have $R_0(\theta)=(\theta - \alpha) (\theta)_{q-1} $, since the indicial equation is given by $R_0(\lambda)=0$.

Since we sum on $m\geq 0$, we notice that, for $j\leq q-1$, $R_j(\theta)=\sum_{k=q-1-j}^q c_{j,k} (\theta)_k$, for suitable constants $c_{j,k}$. So $R_j$ is divisible by $(\theta)_{q-1-j}$, which is non-trivial precisely for $j\leq q-2$.

Assume for now that the width satisfies $r\geq q-1$ (which is the case if $L$ is Fuchsian at $\infty$, so $\deg(P_q)\geq q$, so $r=\deg(\widetilde{P}_q)\geq q-1$). Then for $j\leq q-2\leq r-1$, $(\theta)_{q-1-j}$ divides $R_{j}(\theta)$, hence $R_j(m)=0$ for $m=0, 1, \dots, q-2-j$. This means that $R_{j-k}(k)=0$, for all $0\leq j \leq q-2$ and $0\leq k \leq j$. So the first $(q-2)$ initial conditions in (\ref{rec-R-ic}) are all automatically satisfied by any $(x_0, x_1, \cdots, x_{q-2})$.

Now $x_{q-1}$ satisfies $R_0(q-1) x_{q-1} + \text{expression}(x_0, x_1, \cdots, x_{q-2})=0$. Since $R_0(q-1)=(q-1-\alpha) (q-1)!\neq 0$, $x_{q-1}$ is uniquely determined by the first $(q-1)$ terms. Similarly, for any $n\geq q-1$, $x_n$ will satisfy $R_0(n) x_{n} + \text{expression}(x_0, x_1, \cdots, x_{n-1})=0$ (where the expression depends actually only on the previous $r$ terms). Since $R_0(n)=(n-\alpha) (n)_{q-1}\neq 0$ by our assumption on $\alpha$, $x_{n}$ is uniquely determined by the previous terms.

Because we can take $x_0, x_1, \cdots, x_{q-2}$ to be arbitrary, the space of analytic solutions has dimension at least $(q-1)$ and by the remark above is exactly $(q-1)$.

If the width satisfies $r\leq q-2$ then the proof above shows that all initial conditions and the first full conditions up to $(q-2)$ are satisfied by all the choices of $x_0, x_1, \cdots, x_{q-2}$, since all the involved coefficients are all zero. The rest of the proof follows as above.
\end{proof}

\section{Self-adjointness and symmetries}\label{sec-4-SA-ASA}
In $\mathbb{C}(z)\langle D\rangle$, we define the {\it the formal adjoint} by the rules $D^*=-D$, $z^*=z$, and $\adj{(AB)}=B^* A^*$. It follows that
$$
\adj{(P_k(z) D^k)}=(-D)^k P_k(z) = (-1)^k P_k(z) D^k + \text{ lower order terms}.
$$
Hence if $L=\sum_{k=0}^q P_k(z) D^k$ one gets
$$
\adj{L}=\sum_{k=0}^q P_k^{\sharp}(z) D^k, \text{ where }
P_q^{\sharp}=(-1)^q \, P_q.
$$
We call $L$ {\it formally self-adjoint} if $\adj{L}=(-1)^q \, L$. We notice that this depends on the specific form of $L$ and is not an invariant of the differential module induced by $L$. However, there is a wider definition of self-adjointness at the differential module level involving a non-degenerate bilinear form invariant under monodromy on the local spaces of solutions, but we are not going to use that.

We notice that $\adj{\theta}=-\gt-1$ and, as operators, $\gt z^j = z^j (\gt +j)$. Hence, if $L_\gt=\sum_{k=0}^r z^k R_k(\gt)$, then
$\adj{L_\gt}=\sum_{k=0}^r z^k R_k(-\gt-k-1)$.

\begin{lemma}\label{L-SA-symmetry}
If $L=\sum_{k=0}^q P_k(z) D^k$ is in CPF and $z=0$ is a singularity of order $m$, so $L_\gt=z^{q-m}\, L$. Then $L$ is self-adjoint if and only if

\begin{equation}\label{eq-f-symmetry} 
R_k(-\gt-k-1)=(-1)^q R_k(\gt+q-m), \text{ for all } 0\leq k \leq r.
\end{equation}
\end{lemma}

\noindent
\begin{remark}
We can express this as the fact that $R_k$ are polynomials with center of symmetry at $(q-m-1)/2-k/2$ and are even or odd with respect to this center of symmetry depending on the parity of the order $q$. Since $Q_k=R_{r-k}$ this is also equivalent to the following symmetry of the recurrence polynomials:

\begin{equation}\label{eq-r-symmetry} 
Q_k(-\gt-r+k)=(-1)^q Q_k(\gt+q-m-1), \text{ for all } 0\leq k \leq r.
\end{equation}
\end{remark}

\vspace{.2cm}
\begin{proof}
We notice that $z^{m-q} \, L_\gt \, z^{q-m}= \sum_{k=0}^r z^k \, R_k(\gt+q-m)$. Since $L_\gt=z^{q-m} L$ it follows that $\adj{L}=(-1)^q L$ if and only if
$\adj{L_\gt}=(-1)^q z^{m-q} \, L_\gt \, z^{q-m}$.
Using the formulas above for the two sides and identifying the coefficients of $z^k$, we obtain the required identity.
\end{proof}

For us the two important cases will be $m=q-1$ and $m=1$. If $m=q-1$ the symmetries become
$$
R_k(-\gt-k)=(-1)^q R_k(\gt) \text{ and }
Q_{k}(-\gt-r+k)=(-1)^q Q_{k}(\gt), \text{ for all } 0\leq k \leq r.
$$

\vspace{.2cm}
If we denote by \(^{\ddagger}\) the {\it anti-involution} of \(\mathbf C[t] \langle \theta \rangle \) defined by
\begin{equation}\label{eq-anti-involution}
    t^{\ddagger}=t,\qquad \theta^{\ddagger}=-\theta,\qquad (AB)^{\ddagger}=B^{\ddagger}A^{\ddagger},
\end{equation}
we notice that
$L_\gt^{\ddagger}=\sum_{k=0}^r z^k R_k(-\gt-k)$. So the $R_k$-symmetries above are equivalent to $L_\gt^{\ddagger}=(-1)^q L_\gt$. Hence  $\adj{L}=(-1)^q L$ if and only if  $L_\gt^{\ddagger}=(-1)^q L_\gt$. We also notice that $Q_r$ (or $R_0$) is symmetric with respect to $0$ (that is, even or odd, depending on the parity of $q$) and $Q_0$ will be symmetric with respect to $-r/2=-(\text{degree}(P_q)-q+1)/2$, which is an invariant of $L$, as expected, since $Q_0$ does not depend on which finite non-zero singularity we base the analysis at.

In the case $m=1$, $Q_r$ is symmetric around $-(q-2)/2$. So if $z=a$ is a simple singularity, then, when we translate it to 0, the symbols of $a$ which are the roots of the corresponding $Q_r$ are invariant under the transform $\lambda \mapsto q-2-\lambda$.

Using Lemma \ref{L-symbols}, part 3, it follows that if $L$ is formally self-adjoint of order $q$ every finite simple singularity has symbol $\{0, 1, \dots, q-2, (q-2)/2\}.$


\vspace{.5cm}
\begin{lemma}\label{l-SA-factorization}
Let $L \in \mathbb{C}(z)\langle D\rangle$ be a linear differential operator and suppose $L = A B$ is a factorization.
 If $L$ is self-adjoint of order $q$, then in any nontrivial factorization $L=AB$ one may assume that $B$ has order at most
$\displaystyle{
\left\lfloor \frac{q}{2} \right\rfloor.
}$
\end{lemma}

\begin{proof}
We have $\text{order}(A)+\text{order}(B)=q$, so the minimum of the two orders is at most $\lfloor \frac{q}{2} \rfloor$. Since $L$ is self-adjoint, $L=(-1)^q \adj{B} \adj{A}$ while the minimum of the order of $B$ or $\adj{A}$ is at most $\lfloor \frac{q}{2} \rfloor$ by the above.
\end{proof}

\section{The main results}\label{sec-Main}

Recall that we are interested in the sequences
$$
A_{n}^{(d)} =\sum_{n_1+\dots+n_d=n} \frac{(2n)!}{(n_1!)^2 (n_2!)^2 \dots (n_d!)^2}
\quad \text{ and } \quad
x_{n}^{(d)} = \frac{A_{n}^{(d)}}{\binom{2n}{n}} = \sum_{n_1+\dots+n_d=n} \frac{(n!)^2}{(n_1!\, n_2! \dots n_d!)^2}.
$$

\vspace{.2cm}
\begin{theorem} \label{T-LF-properties}
We denote by
$\displaystyle{
F_d(z)=\sum_{n \ge 0} x_n^{(d)} \, z^n
}$
the generating function of the sequence $(x_n^{(d)})_n$. For $d\ge 3$, we have:
\begin{enumerate}
\item $F_d(z)$ is annihilated by a differential operator
$L_{d-1}=L_{d-1,F}\in \mathbb{C}(z)\langle D\rangle$  of order $d-1$.

\item $L_{d-1}$ is Fuchsian and has maximally unipotent monodromy (MUM) at $z=0$.

\item The canonical polynomial form of $L_{d-1}$ is formally self-adjoint.

\item The finite singularities of $L_{d-1}$ away from $0$ are all simple and form the set
\[
\left\{ \frac{1}{k^2} \ : \ k\in \mathbb{Z}_{\ge 1},\ k \equiv d \pmod{2} \right\},
\]
so there are $\lfloor \frac{d+1}{2} \rfloor$ such.

\item At each nonzero finite singularity $z=\frac{1}{k^2}$ the local exponents are
\[
0,1,2,\dots,d-3, \frac{d-3}{2}.
\]

\item At infinity, the local exponents are
\[
\begin{cases}
\{1,1,2,2,\dots,n,n\}, & \text{ if } d=2n+1 \ \text{(odd)},\\[4pt]
\{1,1,2,2,\dots,(n-1),(n-1),n/2\}, & \text{ if } d=2n \ \text{(even)}.
\end{cases}
\]

\item At each finite nonzero singularity there are $d-2$ linearly independent analytic solutions.

\item $F_d(z)$ is the distinguished analytic solution at $z=0$ and
is singular at every nonzero finite singularity of $L_{d-1}$ (there is no domain $\Omega$ containing $0$ and some $1/k^2$ and an analytic $\tilde{F}_d$ extending $F_d$ to $\Omega$).
\end{enumerate}
\end{theorem}

\vspace{.2cm}
\begin{remark}
The theorem holds for $d=2$ if we interpret MUM as no-singularity (symbol 0 at $z=0$). In this case, the symbol at the only finite singularity $z=1/4$ is $-1/2$, while at $\infty$ the symbol is $1/2$, which is the interpretation of parts 5 and 6 of the theorem in this case.
\end{remark}

\begin{remark}
Part~1 of the theorem answers in the affirmative the conjecture contained in Remark~3.5 of \cite{DS1}.
\end{remark}

\vspace{.2cm}
\begin{cor} \label{cor-F-irreductibilty}
The operator $L_{d-1,F}$ above is irreducible over $\mathbb{C}(z)\langle D\rangle$.
\end{cor}

\noindent
{\it Proof of Corollary~1.}
Denote $L_{d-1,F}=L$ and assume $L=AB$. By Part 3 of Theorem \ref{T-LF-properties} and Lemma \ref{l-SA-factorization}, we can assume that $\text{order}(B)\leq \lfloor \frac{d-1}{2} \rfloor$. By Lemma \ref{L-factorization}, $B$ is Fuchsian. Every finite non-zero singularity $z=a$ is either a singularity of $L$ with $\mathrm{Sol}_a(B) \subseteq \mathrm{Sol}_a(L)$ or is an apparent singularity.

If $d=2n+1$, for $n\geq 1$, the symbols of $L$ at all finite non-zero singularities are $0, 1, 2, \dots, 2n-2, n-1$. $B$ has order $m\leq n$ and its symbols at these points are a subset of $m$ elements of the above. The smallest possible collection is $0,1,2,\dots,m-1$, with sum $m(m-1)/2$. At any apparent singularity the symbol is a collection of $m$ distinct non-negative integers (Lemma \ref{L-factorization}), so the sum is also at least $m(m-1)/2$. At $z=0$ and $\infty$ the symbols of $B$ are non-negative (subsets of the symbols of $L$ at the respective points), so by Lemma \ref{L-general} they must be all 0, which contradicts Part 6 of Theorem \ref{T-LF-properties}.

If $d=2n$, for $n\geq 2$, the symbols of $L$ at all finite non-zero singularities are $0, 1, 2, \dots, 2n-3, (2n-3)/2$. $B$ has order $m\leq n-1$ and its symbols at these points are a subset of $m$ elements of the above. Since $m-1\leq n-2 < (2n-3)/2$, the smallest possible collection is still $0, 1, 2, \dots, m-1$, with sum $m(m-1)/2$. At any apparent singularity the symbol is a collection of $m$ distinct non-negative integers (Lemma \ref{L-factorization}), so the sum is also at least $m(m-1)/2$.
We obtain again a contradiction, as explained in the paragraph above.

It follows that $L$ is irreducible. \hfill $\square$

\vspace{.2cm}
\noindent
\begin{remark}
Corollary~1 holds for $F_2$ as well since an operator of order 1 is automatically irreducible.
\end{remark}


\vspace{.2cm}
\begin{theorem} \label{T-LA-properties}
We denote by
$\displaystyle{
A_d(z)=\sum_{n \ge 0} A_{n}^{(d)} \, z^n
}$
the generating function of the sequence $(A_n^{(d)})_n$. For $d\ge 2$, we have:
\begin{enumerate}
\item $A_d(z)$ is annihilated by a differential operator
$L_{d}=L_{d,A}\in \mathbb{C}(z)\langle D\rangle$  of order $d$.

\item $L_{d}$ is Fuchsian and has maximally unipotent monodromy (MUM) at $z=0$.

\item The canonical polynomial form of $L_{d}$ is formally self-adjoint.

\item The finite singularities of $L_{d}$ away from $0$ are all simple and form the set
\[
\left\{ \frac{1}{4k^2} \ : \ k\in \mathbb{Z}_{\ge 1},\ k \equiv d \pmod{2} \right\},
\]
so there are $\lfloor \frac{d+1}{2} \rfloor$ such.

\item At each nonzero finite singularity $z=\frac{1}{4k^2}$ the local exponents are
\[
0,1,2,\dots,d-2, \frac{d-2}{2}.
\]

\item At infinity, the local exponents are
\[
\begin{cases}
\{ \frac12,1,\frac32,\dots,\frac{d-1}{2}, \frac{d}{2} \}, & \text{ if } d=2n+1 \ \text{(odd)},\\[4pt]
\{ \frac12,1,\frac32,\dots,\frac{d-1}{2}, \frac{d}{4} \}, & \text{ if } d=2n \ \text{(even)}.
\end{cases}
\]

\item At each finite nonzero singularity there are $d-1$ linearly independent analytic solutions.

\item $A_d(z)$ is the distinguished analytic solution at $z=0$ and
is singular at every nonzero finite singularity of $L_{d}$ (there is no domain $\Omega$ containing $0$ and some $1/(4k^2)$ and an analytic $\tilde{A}_d$ extending $A_d$ to $\Omega$).
\end{enumerate}
\end{theorem}

\vspace{.2cm}
\begin{remark}
$A_1=F_2$, so the theorem holds for $d=1$ with the same caveats as for $F_2$.
\end{remark}

\vspace{.2cm}
\begin{cor} \label{cor-A-irreductibilty}
Then operator $L_{d,A}$ above is irreducible over $\mathbb{C}(z)\langle D\rangle$.
\end{cor}

\noindent
{\it Proof of Corollary~2.}
Denote $L_{d,A}=L$ and assume $L=AB$. By Lemma \ref{l-SA-factorization} and part 3 of Theorem \ref{T-LF-properties}, we can assume that $\text{order}(B)\leq \lfloor \frac{d}{2} \rfloor$. By Lemma \ref{L-factorization}, $B$ is Fuchsian. Every finite non-zero singularity $z=a$ is either a singularity of $L$ with $\mathrm{Sol}_a(B) \subseteq \mathrm{Sol}_a(L)$ or is an apparent singularity.

If $d=2n$, for $n\geq 1$, the symbols of $L$ at all finite non-zero singularities are $0, 1, 2, \dots, 2n-2, n-1$. $B$ has order $m\leq n$ and its symbols at these points are a subset of $m$ elements of the above. The smallest possible collection is $0,1,2,\dots,m-1$, with sum $m(m-1)/2$. At any apparent singularity the symbol is a collection of $m$ distinct non-negative integers (Lemma \ref{L-factorization}), so the sum is also at least $m(m-1)/2$. At $z=0$ and $\infty$ the symbols of $B$ are non-negative (subsets of the symbols of $L$ at the respective points), so by Lemma \ref{L-general} they must be all 0, which contradicts Part~6 of Theorem~\ref{T-LA-properties}.

If $d=2n+1$, for $n\geq 1$, the symbols of $L$ at all finite non-zero singularities are $0, 1, 2, \dots, 2n-1, (2n-1)/2$. $B$ has order $m\leq n$ and its symbols at these points are a subset of $m$ elements of the above. Since $m-1\leq n-1 < (2n-1)/2$, the smallest possible collection is still $0, 1, 2, \dots, m-1$, with sum $m(m-1)/2$. At any apparent singularity the symbol is a collection of $m$ distinct non-negative integers (Lemma \ref{L-factorization}), so the sum is also at least $m(m-1)/2$.
We obtain again a contradiction, as explained in the paragraph above.

It follows that $L$ is irreducible. \hfill $\square$

\begin{remark}
Corollary~2 holds for $A_1=F_2$ as well since an operator of order 1 is automatically irreducible.
\end{remark}

\section{Proofs}\label{sec-proofs}
In this section we use an algebraic and combinatorial analysis of the symmetric powers of a modified Borel transform of $F_d$ to prove Parts~1 -- 6 of Theorem~\ref{T-LF-properties}. In the process, we also prove that the sequences  $(x_n^{(d)})_n$ satisfy a $P$-recurrence of width $r_d=r=\lfloor \frac{d+1}{2} \rfloor$, so of the form
$$
Q_r(n+r)x_{n+r}+\cdots + Q_0(n)x_n=0, \; n \ge 0.
$$
Here we show that $Q_r(n)=n^{d-1}$, compute explicitly $Q_{r-1}$ for $d \geq 3$, and show that
\begin{equation}\label{Q0n-odd}
Q_0(n)=(-1)^{r}
\cdot
\left(
 \prod_{\substack{1\le k\le d\\ k\equiv d\!\!\!\pmod 2}} k^2
\right)
\cdot
 \prod_{k=1}^{\tfrac{d-1}{2}} (n+k)^2,
\text{ if } d \geq 3 \text{ odd},
\end{equation}

and

\begin{equation}\label{Q0n-even}
Q_0(n)=(-1)^{r}
\cdot
\left( \prod_{\substack{1\le k\le d\\ k\equiv d\!\!\!\pmod 2}} k^2 \right)
\cdot \left( n+\frac{d}{4} \right)
\cdot
 \prod_{k=1}^{\tfrac{d-2}{2}} (n+k)^2,
\text{ if } d \geq 4 \text{ is even}.
\end{equation}
(Note that for $d=2$ we have  $r=1$,  $Q_1(n)=n$, $Q_0(n)=-4(n+\frac{1}{2})$.)

Part~7 is an immediate consequence of Parts~4 and 5, and Lemma~\ref{L-structure}.
Part~8 has essentially been proven in our earlier paper \cite{DS1}, but we will give a quick recap here using the inductive definition of $F_d$ by the Hadamard convolution method from that paper.

From Part~8 it immediately follows that the $P$-recurrence above has minimal width, which is precisely the number of finite nonzero singularities of $L_{d-1,F}$. This is because those singularities are then true singularities of $F_d$, so the framework dictionary in Section~\ref{sec-3-framework} precludes a recurrence with strictly lower width.

While essentially the same proof works for $A_d$ and Theorem~\ref{T-LA-properties}, with the appropriately modified Borel transform, we conclude the section with a direct derivation of Theorem~\ref{T-LA-properties} from Theorem~\ref{T-LF-properties} and the framework dictionary, as we think this better emphasizes the intimate relation between the two sequences, which is the theme of our two papers in the field.

\subsection{The (modified) Borel Transform of $F_d$}

Since
$$
x_n^{(d)}=\sum_{k_1+ \cdots + k_d=n}\frac{(n!)^2}{(k_1! \cdots k_d!)^2}
$$
we define the modified Borel Transform of $F$ to be the entire power series
$$
Y_d(t)=\sum_{n \ge 0} \frac{x_n^{(d)}}{(n!)^2} \, t^n
 = \sum_{n \ge 0} y_n^{(d)} \, t^n, \text{ with }
 y_n^{(d)}=\frac{x_n^{(d)}}{(n!)^2}.
$$
Note that if $$\phi(t)=\sum_{m\ge 0}\frac{t^m}{(m!)^2}$$ we immediately get that $Y_d=\phi^d$.

Now if as usual $\theta=t\frac{d}{dt}$, we have
$(\theta^2-t)\phi=0$.
In particular, with the notation $u=\phi$ and $v=\theta \phi$, it follows that $Y_d$ lives in the homogeneous space generated by $u^{d-k}v^{k}$, for $0\leq k \leq d$, under the generating relations
$$
\theta u=v \quad\text{ and }\quad \theta v=tu.
$$
For a fixed integer \(d\ge 1\), define
\[
e_j:=u^{\,d-j}v^{\,j}\qquad (\text{for } 0\le j\le d).
\]
The differential operator $\theta$ acts on it (by the Leibniz rule) as follows:
\[
\theta e_j=(d-j)e_{j+1}+jt\,e_{j-1},
\]
for $0\le j\le d$, with the convention \(e_{-1}=e_{d+1}=0\).

Consider a sequence of numbers $c_1, c_2, \dots, c_n, \dots$. Let $D_n$ be a length $n$ domino which we tile with two pieces: a one cell which we tag with $\theta$, and a 2-cell which we tag with $-c_k t$ when it is placed on the position $k,k+1$, where $1\leq k \leq n-1$. For each such tiling $T$ we associate the differential operator $P_T$ obtaining by multiplying the tags of the tiles from right to left, and we define
\begin{equation}\label{eqn-tile-recurrence}
    P_n(c_1, \dots, c_{n-1})= \sum_T P_T.
\end{equation}
For example, $P_0=1$, $P_1=\theta$, $P_2(c_1)=\theta^2-c_1 t$, and $P_3(c_1,c_2)=\theta^3-c_1 \theta t-c_2 t \theta$.

\vspace{.4cm}
From Section~\ref{sec-4-SA-ASA}, we recall that \(^{\ddagger}\) is the anti-involution  of \(\mathbf C[t] \langle \theta \rangle \) defined by (\ref{eq-anti-involution}).

\vspace{.3cm}
\begin{lemma}[Tilling differential operators]
For $n\geq 2$, we have:

$$
P_{n+1}(c_1, \dots, c_{n}) = \theta P_n(c_1, \dots, c_{n-1}) -
 c_n t P_{n-1}(c_1, \dots, c_{n-2});
$$

$$
P_{n+1}(c_1, \dots, c_{n}) = P_n(c_2, \dots, c_{n}) \theta -
 c_1 P_{n-1}(c_3, \dots, c_{n}) t;
$$

\begin{equation}\label{eq:adjoint-P}
P_n^{\ddagger}(c_1, \dots, c_{n-1}) = (-1)^n
P_n(c_{n-1}, \dots, c_{1}).
\end{equation}
\end{lemma}

\begin{proof}
The first relation follows by looking at the last tag in a tiling $T$ of $D_n$, since that can be either $\theta$ or $-c_n t$. The second relation follows from looking at the first tag in a tiling $T$ of $D_n$, since that can be either $\theta$ or $-c_1 t$. $P_0=1$, $P_1=\theta$, $P_2(c_1)=\theta^2-c_1 t$ manifestly satisfy $P_n^{\ddagger} = (-1)^n P_n$. Assuming inductively the relation~(\ref{eq:adjoint-P}), we get

$$
\begin{aligned}
 P_{n+1}^{\ddagger}(c_1, \dots, c_{n}) & = - P_n^{\ddagger}(c_1, \dots, c_{n-1}) \theta - c_n P_{n-1}^{\ddagger}(c_1, \dots, c_{n-2}) t \\
 & = (-1)^{n+1}\left(  P_n(c_{n-1}, \dots, c_{1}) \theta
 - c_n P_{n-1}(c_{n-2}, \dots, c_{1}) t
\right) \\
 & = (-1)^{n+1} P_{n+1}(c_{n}, \dots, c_{1}).
\end{aligned}
$$
\end{proof}

\vspace{.2cm}
Consider the family of operators \(P_{d,n}\in \mathbb{C}[t]\langle \theta\rangle\) given recursively by
$P_{d,0}=1$, $P_{d,1}=\theta$, and for $1\le n\le d$,
\[
P_{d,n+1}=\theta P_{d,n}-n(d-n+1)\,t\,P_{d,n-1}.
\]
This is an example of the previous construction:
$$
P_{d,n+1}= P_{n+1}(d, 2(d-1), \cdots, n (d-n+1)), 1\leq n \leq d.
$$

\vspace{.2cm}
\begin{lemma}[Symmetric power recursion] For a fixed integer \(d\ge 1\) and $e_j=u^{\,d-j}v^{\,j}$, for $0\le j\le d$, the following hold:
\begin{enumerate}
    \item $P_{d,d+1}(e_0)=0$.
    \item $P_{d,d+1}^{\ddagger}=(-1)^{d+1} P_{d,d+1}.$
\end{enumerate}
\end{lemma}

\begin{proof}
We prove by induction on \(n\) that
\[
P_{d,n}(e_0) = c_{d,n}\,e_n, \text{ for } 0\le n\le d
\text{ and where } c_{d,n}=\frac{d!}{(d-n)!}.
\]
It is clear that
$$
P_{d,0}(e_0) = e_0 \text{ and } P_{d,1}(e_0) = \theta \, e_0=d\,e_1.
$$

\noindent
Assume that for some \(1\le n\le d-1\),
\[
P_{d,n}(e_0)=c_{d,n}\, e_n
\text{ and }
P_{d,n-1}(e_0)=c_{d,n-1}\, e_{n-1}.
\]
Then
\[
P_{d,n+1}(e_0)
=
\theta P_{d,n}(e_0)-n(d-n+1)t\,P_{d,n-1}(e_0),
\]
hence
\[
P_{d,n+1}(e_0)
=
c_{d,n}\, \theta \, e_n-n(d-n+1)t\,c_{d,n-1}\, e_{n-1}.
\]
Using
\[
\theta \, e_n=(d-n)e_{n+1}+nt\,e_{n-1},
\]
we obtain
\[
P_{d,n+1}(e_0)
=
c_{d,n}(d-n)e_{n+1}
+
n t\bigl(c_{d,n}-(d-n+1)c_{d,n-1}\bigr)e_{n-1}.
\]
But
\[
c_{d,n}=(d-n+1)c_{d,n-1},
\]
so the coefficient of \(e_{n-1}\) vanishes, and therefore
$$
P_{d,n+1}(e_0)=c_{d,n}(d-n) \, e_{n+1} = c_{d,n+1} \, e_{n+1}.
$$
This proves the claim for all \(0\le n\le d\).

Finally, for \(n=d\),
\[
P_{d,d}(e_0)=d!\,e_d.
\]
Hence
\[
P_{d,d+1}(e_0)
=
\theta(d!e_d)-d\cdot 1\cdot t\,P_{d,d-1}(e_0).
\]
Since
\[
\theta e_d=d t\,e_{d-1}
\quad\text{and}\quad
P_{d,d-1}(e_0)=d!\,e_{d-1},
\]
we obtain
\[
P_{d,d+1}(e_0)=d!\,d t\,e_{d-1}-d t\,d!\,e_{d-1}=0.
\]
This finishes the proof of part~1. Part~2 follows from (\ref{eq:adjoint-P}) and the fact that the sequence
$ d, 2(d-1), \cdots, n (d-n+1)$ is palindromic for $n=d$.
\end{proof}

\vspace{.4cm}
\begin{cor}
Let \(Y_d(t)=\phi(t)^d\). Then $P_{d,d+1}(Y_d)=0$.
\end{cor}


\vspace{.4cm}
\begin{proposition}
Fix any integer $d\geq 1$ and
let $ S_d:=\{k^2:\ 1\le k\le d,\ k\equiv d \pmod 2\}. $
Then
\[
P_{d,d+1}
=
\sum_{m=0}^{\lfloor (d+1)/2\rfloor}
(-1)^m t^m
\Bigl(
e_m(S_d)\,\theta^{\,d+1-2m}
+\text{ lower powers of }\theta
\Bigr),
\]
where \(e_m(S_d)\) denotes the \(m\)-th elementary symmetric sum of the elements of \(S_d\).
\end{proposition}

\begin{proof}
We compare only the highest power of \(\theta\) occurring in each coefficient of \(t^m\), which we call $a_{d,n,m}$ for $P_{d,n}$.
From
$$
P_{d,n+1}=\theta P_{d,n}-n(d-n+1)\,t\,P_{d,n-1}
$$
and
$$
\gt \big( t^m \, R_m(\gt) \big) = t^m \gt \, R_m(\gt) + m t^m \, R_m(\gt),
$$
the leading \(\theta\)-coefficient \(a_{d,n,m}\) satisfies
\[
a_{d,n+1,m}=a_{d,n,m}-n(d-n+1)\,a_{d,n-1,m-1}.
\]
The initial values are $a_{d,0,0}=1$, $a_{d,1,0}=1$,
and \(a_{d,n,m}=0\) if \(m<0\) or \(2m>n\).

\vspace{.2cm}
We introduce the following polynomials in variables $X$ and $t$:
$$
U_0(X,t)=1, \; U_1(X,t)=X,
\text{ and }
U_n(X,t)=\sum_{m} a_{d,n,m} \, X^{n-2m} \, t^m.
$$
Also set
\[
\lambda_n:=n(d-n+1)\qquad (n=1,\dots,d).
\]
Then these polynomials satisfy the recurrence relation
$$
U_{n+1} = X \, U_n-\lambda_n t U_{n-1}.
$$
Hence $U_n$'s satisfy the same recurrence as the $P_n$'s, except that now $X$ and $t$ are commuting variables. In particular, $U_n (\lambda_1, \lambda_2, \dots, \lambda_{n-1})$ is given by equation~(\ref{eqn-tile-recurrence}), where now the order of the multiplication of the tags is irrelevant. So

\begin{equation}\label{eq:cell-U}
U_{d+1}(X,t) = \sum_m (-1)^m \sum_{T, \#2-cells(T)=m}
 \left( \prod_{(j,j+1) \text{ cell of } T} \lambda_j \, X^{d+1-2m}\, t^m \right).
\end{equation}

Let $s^2=t$ and $B_d(s)$ be the Sylvester-Kac tri-diagonal $(d+1)\times (d+1)$-matrix \cite{Kac47} given by
$$
B_d(s) =
\begin{bmatrix}
0 & ds & 0 & \cdots & 0 & 0\\
s & 0 & (d-1)s & \cdots & 0 & 0\\
0 & 2s & 0 & \ddots & 0 & 0\\
\vdots & \vdots & \ddots & \ddots & 2s & 0\\
0 & 0 & 0 & (d-1)s & 0 & s\\
0 & 0 & 0 & 0 & ds & 0
\end{bmatrix}
$$
Expanding its characteristic polynomial $\det(XI_{d+1}-B_d(s))$ on rows it follows that this is exactly our polynomial $U_{d+1}(X,t)$.
Since the eigenvalues of $B_d(s)$ are $\alpha_j = (d-2j) s$, $0\leq j\leq d$, it follows that $U_{d+1}(X,t) = \prod_{j=0}^d (X-\alpha_j)$. Grouping the factors in pairs $\pm k$, $k\in S_d$, for $d$ odd,  and in pairs $\pm k$ plus an extra $X$ factor, for $d$ even, it follows that
$$
U_{d+1}(X,t) = \prod_{k\in S_d} (X^2-k^2 t),  \text{ for $d$ odd},
$$
and
$$
U_{d+1}(X,t) = X \prod_{k\in S_d} (X^2-k^2 t),  \text{ for $d$ even}.
$$
This is what we wanted to prove.
\end{proof}

\vspace{.2cm}
\begin{remark}\label{R-R1Btheta}
Conform our framework in Section~\ref{sec-3-framework}, equation (\ref{eqn-Rs}), let $P_{d,d+1} = \sum_{m=0}^{\lfloor (d+1)/2 \rfloor}
 t^m R_m^{B}(\theta)$. Then $R_0^B(\theta)=\theta^{d+1}$ and it can be proved that

\begin{equation}\label{eqn-R1Btheta}
R_1^B(\theta) =\sum_{j=0}^{d-1}(d-j)\binom{d+2}{j} \theta^j
= (\theta+1)^{d+1}(d-2\theta)+\theta^{d+1}(d+2+2\theta).
\end{equation}
$R_1^B$ has degree $d-1$ and $R_1^B(0)=d$.
\end{remark}

\begin{remark}
Since $P_{d,d+1}^{\ddagger}=(-1)^{d+1} P_{d,d+1}$, $R_m^B(\gt)$ are symmetric with respect to $-m/2$. See the discussion after Lemma~\ref{L-SA-symmetry}. Let $r={\lfloor (d+1)/2\rfloor}$. Per our framework, denoting
$ Q_k^B= R_{r-k}^B $, $Q_0^B$ is symmetric with respect to $-r/2$.
If $d$ is odd, then $\deg(Q_0^B)=d+1-2r=0$, so
$$
Q_0^B=(-1)^{r} e_r(S_d) = (-1)^{r} \prod_{k=1, k\equiv d\!\!\! \pmod 2}^d  k^2.
$$
If $d$ is even, then $\deg(Q_0^B)=d+1-2r=1$, so by the symmetry above
$$
Q_0^B=(-1)^{r} e_r(S_d) (\gt+r/2).
$$
\end{remark}

\vspace{.2cm}
Putting together all the results of this subsection, we obtain the following:

\begin{proposition}\label{P:y-recurrence}
The sequence $(y_n)_n=(y_n^{(d)})_n$ satisfies the recurrence
\begin{equation}\label{eq:y-recurrence}
 \sum_{k=0}^r Q_k^B(n+k) \, y_{n+k}=0
\end{equation}
and the initial conditions
$$
\sum_{k=0}^m  Q_{r-m+k}^B(k) \, y_{k}= 0, \text{ for }0\leq m \leq r-1.
$$
Also
\begin{enumerate}
    \item $\deg({Q_k^B})=d+1+2k-2r$;
    \item $Q_{k}^B(-n-r+k)=(-1)^{d+1} Q_{k}^B(n), \text{ for all } 0\leq k \leq r$.
    \item The leading coefficient of $Q_k^B(n)$ is $(-1)^{r-k} e_{r-k}(S_d)$.
    \item $Q_r^B(n)=n^{d+1}$.
    \item
    If $d$ is odd, then $Q_0^B(n)=(-1)^{r} e_r(S_d)$.
    If $d$ is even, then $Q_0^B(n)=(-1)^{r} e_r(S_d) (n+r/2)$.
    \item $Q_{r-1}^B(n)=R_1^B(n)$ are as noted in Remark~\ref{R-R1Btheta}.
\end{enumerate}
\end{proposition}

\subsection{Transitioning from $(y_n^{(d)})_n$ to $(x_n^{(d)})_n$}
\label{subS-x-to-y}
Substituting $y_{n+k}=x_{n+k}/((n+k)!)^2$ in equation~(\ref{eq:y-recurrence}), noting that $Q_r^B(n+k)=(n+k)^{d+1}$, simplifying $(n!)^2$, and clearing denominators, we obtain that $(x_n)_n$ satisfies the recurrence
\begin{equation}\label{eq:x-recurrence1}
 \sum_{k=0}^r Q_k(n+k) \, x_{n+k}=0,
\end{equation}
with the initial conditions
$$
\sum_{k=0}^m  Q_{r-m+k}(k) \, x_{k}= 0, \text{ for }0\leq m \leq r-1.
$$
We have
$Q_r(n)=n^{d-1}$,
$Q_{r-1}(n)=Q_{r-1}^B(n)$, and
$Q_{r-k}(n)=\prod_{j=1}^{k-1}(n+j)^2 \, Q_{r-k}^B(n)$,
since in the recurrence~(\ref{eq:x-recurrence1}) the coefficient of $x_{n+k}$ is $Q_k(n+k)$. It follows that the degrees of all $Q_k$ are $d-1$,  $Q_0(n)= (-1)^{r} e_r(S_d) \prod_{j=1}^{r-1}(n+j)^2$, when $d$ is odd, and $Q_0(n)= (-1)^{r} e_r(S_d) (n+r/2) \prod_{j=1}^{r-1}(n+j)^2$, when $d$ is even.

Relation~2 from Proposition~\ref{P:y-recurrence} shows that $Q_k^B$ is symmetric with respect to $-(r-k)/2$, so it is an odd or an even polynomial in $(n+(r-k)/2)$, depending on the parity of $d$. $Q_r$ keeps the 0-symmetry and the same parity, $Q_{r-1}$ is unchanged, while for $2\leq k \leq r$ the multiplier $\prod_{j=1}^{k-1}(n+j)^2$ preserves the symmetry of $Q_{r-k}^B$. And since the multiplier is a square, the global parity of $Q_k(n)$ is the same as that of $Q_k^B(n)$. Hence we proved the following:

\begin{proposition}\label{P:x-recurrence}
The sequence $(x_n)_n=(x_n^{(d)})_n$ satisfies the recurrence
\begin{equation}\label{eq:x-recurrence}
 \sum_{k=0}^r Q_k(n+k) \, x_{n+k}=0
\end{equation}
and the initial conditions
$$
\sum_{k=0}^m  Q_{r-m+k}(k) \, x_{k}= 0, \text{ for }0\leq m \leq r-1.
$$
Also
\begin{enumerate}
    \item $\deg({Q_k})=d-1$;
    \item $Q_{k}(-n-r+k)=(-1)^{d-1} Q_{k}(n), \text{ for all } 0\leq k \leq r$.
    \item The leading coefficient of $Q_k(n)$ is $(-1)^{r-k} e_{r-k}(S_d)$.
    \item $Q_r(n)=n^{d-1}$.
    \item For $0\leq k \leq r-2$,
$$
\prod_{j=1}^{r-k-1}(n+j)^2 \text{ divides } Q_k(n).
$$
    \item
    If $d$ is odd, then $Q_0(n)= (-1)^{r} e_r(S_d) \prod_{j=1}^{r-1}(n+j)^2$.
    If $d$ is even, then $Q_0(n)= (-1)^{r} e_r(S_d) (n+r/2) \prod_{j=1}^{r-1}(n+j)^2$.
    \item $Q_{r-1}(n)=R_1^B(n)$ are as noted in Remark~\ref{R-R1Btheta}.
\end{enumerate}
\end{proposition}

\vspace{.3cm}
If $R_k(\gt)=Q_{r-k}(\gt)$ and $L_\theta =\sum_{k=0}^r z^k R_k(\theta)$ then $L_\theta(F_d)=0$, $L_\theta$ has order $d-1$, it is Fuchsian at $0$ and $\infty$, and it is MUM at $0$. The symbol at $\infty$ is given by $Q_0(-\lambda)=0$, so it is precisely as given in Part~6 of Theorem~\ref{T-LF-properties}.

The symmetry given by Part 2 of Proposition~\ref{P:x-recurrence} shows that $L_\theta^\ddagger=(-1)^{d-1} L_\gt$. If $L$ is in CPF standard form, the MUM condition at 0 shows that $m=q-1$, so $L_\gt=zL$, hence $L$ is formally self-adjoint, $L^*=(-1)^{d-1} L$. See the discussion after Lemma~\ref{L-SA-symmetry}.

If we write $L_{\theta} = \sum_{k=0}^{d-1} \widetilde{R}_k(z) \theta^k$, Remark~\ref{R:spc} shows that the finite non-zero singularities are the solutions of $\widetilde{R}_{d-1}(z)=0$ and
$\widetilde{R}_{d-1}(z)=\sum_{k=0}^r [\theta^{d-1}]R_k \, z^k$. Since the degree of $R_k$ is $d-1$, the coefficients above are the leading coefficients, so they equal $(-1)^k e_k(S_d)$, $0 \leq k \leq r$. Hence
\begin{equation}
 \widetilde{R}_{d-1}(z) = (-1)^r \prod_{n\in S_d} (nz-1).
\end{equation}
This shows that the finite non-zero singularities are simple and are precisely the inverses of the elements in $S_d$, namely $1/k^2$, $1\leq k \leq d$, $k\equiv d \pmod 2$. In particular, $L$ is Fuchsian at these singularities as well, with symbol $\{0, 1, \cdots, d-3,(d-3)/2\}$.

We  proved Parts~1--6 in the Theorem~\ref{T-LF-properties}. Since the extra symbol of each non-zero singularity is $(d-3)/2$, Lemma~\ref{L-structure} shows that Part~7 holds as well.

\subsection{The Hadamard convolution}\label{subS-Hadamard}
From our paper \cite[Equation (17)]{DS1}, we recall that $F_{d+1}$ is given inductively by

\begin{equation} \label{eq:ConvExpN}
F_{d+1}(w) = \frac{1}{2\pi i}\int_{|t|=1/s_d}
    F_d(t)P(t,w)^{-1/2}dt, \text{ for } P(t,w)=w^2-wt(2w+2)+t^2(w-1)^2,
\end{equation}
where $s_d$ is any number satisfying
$s_d<\tfrac{1}{(d+1)^2}$.

We notice that $P(w,t)=P(t,w)$ and $P$ has roots

\begin{equation} \label{eq:roots}
t_1(w)=\frac{w}{(\sqrt w +1)^2} \text{ and }
t_2(w)=\frac{w}{(\sqrt w -1)^2},
\end{equation}
where changing the determination of the square root, just switches them. Also $t_1(0)=t_2(0)=0$, $t_1(\infty)=t_2(\infty)=1$, $t_1(1)=\frac{1}{4}$, $t_2(1)=\infty$, in the usual limit sense in the latter cases (i.e., as $w \to 1$ the roots $t_1(w) \to \frac{1}{4}$, $t_2(w) \to \infty$). It is the case that $t_{1,2}(w) \in [0, \infty]$ if and only if $w \in [0, \infty]$.

Let $b=1/k^2$ be one of the singularities of $L_d$ (the annihilating operator of $F_{d+1})$. Assume first that $b\neq 1/(d+1)^2$ or $b\neq 1$, when $z=1$ is a singularity. Then $b$ comes from the singularities
$a_{k-1}=1/(k-1)^2$ and $a_{k+1}=1/(k+1)^2$ of $F_d$ and the Hadamard deformation as described in \cite{DS1}. Locally, near $b$,
$$
t_1(w)-a_{k+1} = \lambda_1 (w-b)+O((w-b)^2), \text { where } \lambda_1=t_1'(1/k^2)=k^3/(k+1)^3.
$$
Similarly
$$
t_2(w)-a_{k-1} = \lambda_2 (w-b)+O((w-b)^2), \text { where } \lambda_2=t_2'(1/k^2)=k^3/(k-1)^3.
$$
So $\lambda_1$ and $\lambda_2$ are positive. For the case $b=1/(d+1)^2$, only $t_2$ and $a_{k-1}$ are involved, since $1/(k+1)^2$ is not a singularity of $F_d$, so we only have $\lambda_2$ which is positive.

As seen in \cite{DS1}, $b=1$ is a singularity only when one goes from $F_{\text{even}}$ to $F_{\text{odd}}$ and it comes from the singularities $1/4$ and $\infty$. However at infinity the polynomial has a double root ($t_{1,2}\rightarrow 1$ as $w\rightarrow \infty$), which means that $P^{-1/2}$ is a regular function near $\infty$, so it does not contribute to the singularity in the convolution. For this case we only have $\lambda_1$, which is positive.

Assuming inductively the result, by ignoring the local coefficients, the Transfer Theorem for Hadamard composition \cite[Thm.VI.10]{FS} implies that the convolutions behaves as expected (square type singularity goes to logarithmic singularity, and logarithmic singularity goes to square type singularity) and the coefficients are given by positive $\beta$ function values, $\beta_1$ and $\beta_2$, from the $F_d$ singularities $a_{k-1}$ and $a_{k+1}$, respectively. Putting the above together the coefficient of the singular term of the expansion of $F_{d+1}$ at $b$ is
$c_1 \lambda_1 \beta_1 + c_2 \lambda_2 \beta_2$, where $c_1$ and $c_2$ are the coefficients of the singular part of $F_d$ at $a_{k-1}$ and $a_{k+1}$, respectively. (In the special cases we have only one coefficient, as above.) Since $F_2(z)=(1-4z)^{-1/2}$ the initial coefficient of the singular part at $z=1/4$ is positive, hence the coefficients of the singular parts of $F_3$ at $z=1$ and $z=1/9$ are positive. By induction, all the involved coefficients are positive, so $c_1 \lambda_1 \beta_1 + c_2 \lambda_2 \beta_2$ is positive, so $F_{d+1}$ has a singularity at $1/k^2$, for all $1\leq k \leq d+1$, of the same parity as $d+1$. This proves Part~8 of Theorem~\ref{T-LF-properties}.

\subsection{Obtaining the results for $(A_{n}^{(d)})_n$ from those for $(x_{n}^{(d)})_n$}\label{subS-A-from-x}

\begin{proposition}\label{P:A-recurrence}
With $r=\lfloor \frac{d+1}{2} \rfloor$, the sequence $(A_{n})_n=(A_{n}^{(d)})_n$ satisfies the recurrence
\begin{equation}\label{eq:A-recurrence}
 \sum_{k=0}^r Q_k^A(n+k) \, A_{n+k}=0
\end{equation}
and the initial conditions
$$
\sum_{k=0}^m  Q_{r-m+k}^A(k) \, A_{k}= 0, \text{ for }0\leq m \leq r-1.
$$
Also
\begin{enumerate}
    \item $\deg({Q_k^A})=d$;
    \item $Q_{k}^A(-n-r+k)=(-1)^{d} Q_{k}^A(n), \text{ for all } 0\leq k \leq r$.
    \item The leading coefficients of $Q_k^A(n)$ are $(-1)^{r-k} e_{r-k}(S_d^A)$, where $S_d^A=4S_d$.
    \item $Q_r^A(n)=n^{d}$.
    \item For $0\leq k \leq r-1$,
$$
\prod_{j=1}^{r-k}(2n+2j-1) \prod_{j=1}^{r-k-1}(n+j) \text{ divides } Q_k^A(n).
$$
    \item
    If $d$ is odd, then $Q_0^A(n)= (-1)^{r} e_r(S_d^A) \prod_{j=1}^{r}(n+j-\tfrac12) \prod_{j=1}^{r-1}(n+j) $.
    If $d$ is even, then $Q_0^A(n)= (-1)^{r} e_r(S_d^A) (n+r/2) \prod_{j=1}^{r}(n+j-\tfrac12) \prod_{j=1}^{r-1}(n+j)$.
    \item $Q_{r-1}^A(n) = 2(2n+1) R_1^B(n)$ are as computed above.
\end{enumerate}
\end{proposition}

\begin{proof}
Since $x_n=A_n/\binom{2n}{n}$ it follows that
$$
\sum_{k=0}^r Q_k(n+k) \, A_{n+k} \, \frac{(n+k)!^2}{(2n+2k)!}=0,
$$
with the corresponding initial conditions. After simplification we obtain:
$$
\sum_{k=0}^r Q_k(n+k) \, A_{n+k} \,\frac{(n+k)\cdots (n+1)}{2^k (2n+2k-1)\cdots (2n+1)}=0.
$$
Using Part 5 of Proposition~\ref{P:x-recurrence}, we get that $Q_k(n+k) (n+k)\cdots (n+1)$ is divisible by $(n+r-1)\cdots (n+1)$, so we can simplify all the numerators above except an $n+r$ in the $r$th term. Getting rid of the denominators we obtain (\ref{eq:A-recurrence}), with the appropriate initial conditions, where

\begin{equation}\label{eq-QkA}
Q_r^A(n)=n^d \quad \text{ and } \quad
Q_k^A(n) = 2^{r-k} Q_k(n) \frac{\prod_{j=1}^{r-k}(2n+2j-1)}{\prod_{j=1}^{r-k-1}(n+j)}.
\end{equation}
It follows that $\deg(Q_k^A)=\deg(Q_k)+1=d$ and the symmetries in Part~2 hold. Hence the leading coefficients of $Q_k^A(n)$ are $(-1)^{r-k} 4^{r-k} e_{r-k}(S_d)$, which are by definition $(-1)^{r-k} e_{r-k}(S_d^A)$, which proves Part~3. Finally, the equation~(\ref{eq-QkA}) and the Parts~5--7 from Proposition~\ref{P:x-recurrence} imply the remaining Parts~5--7.
\end{proof}

Parts 1--7 of Theorem~\ref{T-LA-properties} follow from Proposition~\ref{P:A-recurrence} in the same way that the corresponding parts of Theorem~\ref{T-LF-properties} follow from Proposition~\ref{P:x-recurrence}. Part~8 of Theorem~\ref{T-LA-properties} follows from Part~8 of Theorem~\ref{T-LF-properties} and the Transfer Theorem from \cite{FS} since $A_d=F_d*(1-4z)^{-1/2}$ so each singularity $1/4k^2$ of $A_d$ comes just from the singularity $1/k^2$ of $F_d$.

\section{Appendix: Explicit formulas ($P$-symbols, recurrences, and the $L$-operators)}\label{sec-Appendix}

Here we collect the essential formulas in dimensions $d=3$ to $8$. We dedicate a subsection to each dimension and within that subsection we omit in the notation the reference to the dimension (we denote $x_n$ instead of $x_n^{(d)}$ etc.) The low dimensional cases ($d=3$, $4$, and $5$) show explicitly the framework used in our paper and emphasize the close relationship between $F_d$ and $A_d$ (between $(x_n)_n$ and $(A_n)_n$).\footnote{In dimensions $d=3$, $4$, and $5$, for a different writing of the recurrence relations and of the ODEs satisfied by the $F_d$'s and the $A_d$'s, see Remarks~4.1 and 4.3 in \cite{DS1}, respectively. For the sequence $(x_n^{(6)})_n$, the recurrence and the ODE appear in \cite[equations (11e) and (12e)]{GP93}.}

As of July 2026, part of the freely available Mathematica package RISCErgoSum,\footnote{RISCErgoSum is a collection of packages created at the Research Institute for Symbolic Computation (RISC), Linz, Austria. See \href{https://www3.risc.jku.at/research/combinat/software/ergosum/index.html}{RISCErgoSum}.} the {\tt GuessMinRE} command was able to obtain the recurrences up till dimension $d=12$, and the {\tt GuessMinDE} command was able to guess the ODE up to dimension $d=8$ for the $(x_n)_n$ sequence and up to dimension $d=7$ for the $(A_n)_n$ sequence.\footnote{Not surprising, because the $F_{d+1}(z)$ and $A_d(z)$ have similar properties.}

A different way to obtain these formulas is to use the explicit form of the three polynomials $Q_0$, $Q_{r-1}$, and $Q_r$, together with the particular symmetries of the rest of the polynomials, as explained in Propositions~\ref{P:x-recurrence} and \ref{P:A-recurrence}, to reduce the size of the linear system satisfied by the coefficients of the recurrence, and consequently making it solvable.

\subsection{The case $d=3$} 

\vspace{.2cm}
\[
\begin{array}{|c|c|c|c|c|c|c|}
\hline
 & \text{Finite singularities} & \text{Exponents at each finite point} & \text{Exponents at } \infty & q & m & r \\
\hline
F & 1, 1/9 & 0, 0 & 1, 1 & 2 & 1 & 2 \\
\hline
A & 1/4, 1/16 & 0, \tfrac12,1 & \tfrac12,1,\tfrac32 & 3 & 2 & 2 \\
\hline
\end{array}
\]

\vspace{.2cm}
\[
\begin{array}{|c|l|}
\hline
 & \text{Recurrences} \\
\hline
F & (n+2)^2 x_{n+2} - (10 (n+1)^2 + 10 (n+1) + 3) x_{n+1} + 9 (n+1)^2 x_{n}=0 \\ & \\
\hline
A &
\begin{aligned}
(n+2)^3 A_{n+2}
  &- 2 (2(n+1)+1)(10 (n+1)^2 + 10 (n+1) + 3) A_{n+1} \\
  &+ 4\cdot 36 \cdot (n+\tfrac12)(n+1)(n+\tfrac32) A_{n}=0
\end{aligned}
 \\
\hline
\end{array}
\]

\vspace{.2cm}
\[
\begin{array}{|c|l|}
\hline
 & \text{The $L_{\gt} = z L$ operator in CPF (with ${R_k(\gt)}$)} \\
\hline
F &
z^2 \cdot [9(\gt+1)^2] -z\cdot [10\,\gt^2+10\,\gt+3]+\gt^2 \\ & \\
\hline
A & z^2 \cdot [4\cdot 36\cdot (\gt+\tfrac12)(\gt+1)(\gt+\frac32)]-z\cdot[2(2\gt+1)(10\,\gt^2+10\,\gt+3)]+\gt^3 \\ & \\
\hline
\end{array}
\]

\vspace{.2cm}
\[
\begin{array}{|c|l|}
\hline
 & \text{The $L_{\gt} = z L$ operator (with $\widetilde{R_k}(z)$)} \\
\hline
F & (z-1)(9z-1) \, \gt^2 + 2z(9z-5) \, \gt +3z(3z-1) \\ & \\
\hline
A &
(4z-1)(36z-1)\,\theta^3 + 12z(36z-5)\,\theta^2
  +4z(99z-8)\,\theta+6z(18z-1) \\ & \\
\hline
\end{array}
\]

\vspace{.2cm}
\[
\begin{array}{|c|l|}
\hline
 & \text{The $L$ operator in standard CPF} \\
\hline
F & z(z-1)(9z-1) D^2 +(27z^2-20z+1) D +3(3z-1) \\ & \\
\hline
A &
\begin{aligned}
z^2 &(4z-1)(36z-1) D^3 + 3 z (288z^2-60z+1) D^2 \\
    &+ (972z^2-132z+1) D+6(18z-1)
\end{aligned}
 \\
\hline
\end{array}
\]

\setcounter{dimension}{3}

\newpage
\subsection{The case $d=4$} 

\vspace{.2cm}
\[
\begin{array}{|c|c|c|c|c|c|c|}
\hline
 & \text{Finite singularities} & \text{Exponents at each finite point} & \text{Exponents at } \infty & q & m & r \\
\hline
F & 1/4, 1/16 & 0, \tfrac12, 0 & 1, 1, 1 &
\text{\thedimension} & 2 & 2 \\
\hline \stepcounter{dimension}
A & 1/16, 1/64 & 0, 1, 1, 2 & \tfrac12,1,1,\tfrac32 &
\text{\thedimension} & 3 & 2 \\
\hline
\end{array}
\]

\vspace{.2cm}
\[
\begin{array}{|c|l|}
\hline
 & \text{Recurrences} \\
\hline
F &
\begin{aligned}
(n+&2)^3 x_{n+2} \\
 & - 2 [2(n+1)+1] [5 (n+1)^2 + 5 (n+1) + 2] x_{n+1} \\
 & + 64 (n+1)^3 x_{n}=0
\end{aligned}
\\
\hline
A &
\begin{aligned}
(n+&2)^4 A_{n+2} \\
 & - 4 [2(n+1)+1]^2 [5 (n+1)^2 + 5 (n+1) + 2] A_{n+1} \\
  & + 256 (2n+1)(n+1)^2(2n+3) A_{n}=0
\end{aligned}
\\
\hline
\end{array}
\]

\vspace{.2cm}
\[
\begin{array}{|c|l|}
\hline
 & \text{The $L_{\gt}=zL$ operator in CPF (with ${R_k(\gt)}$)} \\
\hline
F &
z^2 \cdot [64(\theta+1)^3] -z\cdot
[ 2(2\theta+1)(5\theta^2+5\theta+2) ]+\gt^3 \\ & \\
\hline
A &
z^2\cdot [256 (2\theta+1) (\theta+1)^2 (2\theta+3)]
 - z \cdot [4(2\theta+1)^2 (5\theta^2+5\theta+2)] + \theta^4
 \\ & \\
\hline
\end{array}
\]

\vspace{.2cm}
\[
\begin{array}{|c|l|}
\hline
 & \text{The $L_{\gt}=zL$ operator (with $\widetilde{R_k}(z)$)} \\
\hline
F &
\begin{aligned}
(4z-1)&(16z-1)\,\theta^3 +6z(32z-5)\,\theta^2 \\
  &+6z(32z-3)\,\theta+4z(16z-1)
\end{aligned}
\\
\hline
A &
\begin{aligned}
(16z-1)&(64z-1)\,\theta^4+32z(128z-5)\,\theta^3 \\
 &+4z(1472z-33)\,\theta^2+4z(896z-13)\,\theta+8z(96z-1)
\end{aligned}
\\
\hline
\end{array}
\]

\vspace{.2cm}
\[
\begin{array}{|c|l|}
\hline
 & \text{The $L$ operator in standard CPF} \\
\hline
F &
\begin{aligned}
 z^2 &(4z-1) (16z-1) D^3 + 3 z (128z^2-30z+1) D^2 \\
     &+ (448z^2-68z+1) D + 4(16z-1)
\end{aligned}
\\
\hline
A &
\begin{aligned}
 z^3 &(16z-1) (64z-1) D^4 + 2z^2(5120z^2-320z+3) D^3 \\
 &+ z (25344z^2-1172z+7) D^2 + (14592z^2-424z+1) D + 8(96z-1)
\end{aligned}
 \\
\hline
\end{array}
\]

\newpage
\subsection{The case $d=5$} 

\[
\begin{array}{|c|c|c|c|c|c|c|}
\hline
 & \text{Finite singularities} & \text{Exponents at each finite point} & \text{Exponents at } \infty & q & m & r \\
\hline
F & 1, 1/9, 1/25 & 0, 1, 1, 2 & 1, 1, 2, 2 &
\text{\thedimension} & 3 & 3 \\
\hline \stepcounter{dimension}
A & 1/4, 1/36, 1/100 & 0, 1, \tfrac32, 2, 3 & \tfrac12, 1, \tfrac32, 2, \tfrac52 &
\text{\thedimension} & 4 & 3 \\
\hline
\end{array}
\]
\

\[
\begin{array}{|c|l|}
\hline
 & \text{Recurrences} \\
\hline
F &
\begin{aligned}
(n &+3)^4 \, x_{n+3} \\
     &-[35(n+2)^4+70(n+2)^3+63(n+2)^2+28(n+2)+5] \, x_{n+2} \\
   & +((n+1)+1)^2 \, [259(n+1)^2+518(n+1)+285] \, x_{n+1} \\
   &- 3^2\cdot 5^2\cdot (n+1)^2 (n+2)^2 \, x_{n}=0
\end{aligned}
\\
\hline
A &
\begin{aligned}
(n &+3)^5 \, A_{n+3} \\
 &-2\,[2(n+2)+1]\,[35(n+2)^4+70(n+2)^3+63(n+2)^2+28(n+2)+5]\, A_{n+2} \\
 & +4\,[2(n+1)+1]\,[(n+1)+1]\,[2(n+1)+3]\,
       [259(n+1)^2+518(n+1)+285] \, A_{n+1} \\
 & - 14400\,(n+\tfrac12)(n+1)(n+\tfrac32)(n+2)(n+\tfrac52) \, A_{n}=0.
\end{aligned}
 \\
\hline
\end{array}
\]

\[
\begin{array}{|c|l|}
\hline
 & \text{The $L_{\gt}=zL$ operator in CPF (with ${R_k(\gt)}$)} \\
\hline
F &
\begin{aligned}
z^3 \cdot & [3^2\cdot 5^2\cdot (\theta+1)^2(\theta+2)^2]
  -z^2\cdot [(\theta+1)^2(259\,\theta^2+518\,\theta+285)] \\
 & +z \cdot [35\,\theta^4+70\,\theta^3+63\,\theta^2+28\,\theta+5]
  -\theta^4
\end{aligned}
\\
\hline
A &
\begin{aligned}
z^3 \cdot & [14400(\theta+\tfrac12)(\theta+1)(\theta+\tfrac32)
   (\theta+2)(\theta+\tfrac52)] \\
 & -z^2\cdot [4(2\theta+1)(\theta+1)(2\theta+3)
     (259\,\theta^2+518\,\theta+285)] \\
 & +z \cdot [2(2\theta+1)
       (35\,\theta^4+70\,\theta^3+63\,\theta^2+28\,\theta+5)]
       -\theta^5
\end{aligned}
\\
\hline
\end{array}
\]

\[
\begin{array}{|c|l|}
\hline
 & \text{The $L_{\gt} = z L$ operator (with $\widetilde{R_k}(z)$)} \\
\hline
F &
\begin{aligned}
(z-1)&(9z-1)(25z-1)\, \theta^4
 +2z(675z^2-518z+35)\,\theta^3 \\
 & +z(2925z^2-1580z+63)\,\theta^2+4z(675z^2-272z+7)\,\theta \\
 & +5z(180z^2-57z+1)
\end{aligned}
\\
\hline
A &
\begin{aligned}
(4z-1)&(36z-1)(100z-1)\, \theta^5
 +10z(10800 z^2-2072 z+35)\,\theta^4 \\
 &+4z(76500 z^2-10205 z+98)\,\theta^3
  +2z(202500 z^2-19790 z+119)\,\theta^2 \\
 &+4z(61650 z^2-4689 z+19)\,\theta
  +10z(5400 z^2-342 z+1)
\end{aligned}
\\
\hline
\end{array}
\]

\[
\begin{array}{|c|l|}
\hline
 & \text{The $L$ operator in standard CPF} \\
\hline
F &
\begin{aligned}
z^3 &(z-1)(9z-1)(25z-1) D^4 + 2z^2 (1350z^3-1295z^2+140z-3) D^3\\
     &+ z (8550z^3-6501z^2+518z-7) D^2 + (7200z^3-3963z^2+196z-1) D\\
     &+ 5(180z^2-57z+1)
\end{aligned}
\\
\hline
A &
\begin{aligned}
 z^4 & (4z-1)(36z-1)(100z-1) \, D^5 
  + z^3 (252000 z^3-62160 z^2+1750 z-10) \, D^4\\
  &+ z^2 (1314000 z^3-268740 z^2+5992 z-25) \, D^3 \\
  &+ z (2295000 z^3-369240 z^2+5964 z-15) \, D^2 \\
  &+ (1080000 z^3-124020 z^2+1196 z-1) \, D 
  +10 (5400 z^2-342 z+1)
\end{aligned}
 \\
\hline
\end{array}
\]

\subsection{The case $d=6$} 

\[
\begin{array}{|c|c|c|c|c|c|c|}
\hline
 & \text{Finite singularities} & \text{Exponents at each finite point} & \text{Exponents at } \infty & q & m & r \\
\hline
F & 1/4, 1/16, 1/36 & 0, 1, 2, 3, \tfrac32 & 1, 1, 2, 2, \tfrac32 &
\text{\thedimension} & 4 & 3 \\
\hline \stepcounter{dimension}
A & 1/16, 1/64, 1/144 & 0, 1, 2, 2, 3, 4
  & \tfrac12,1,\tfrac32,\tfrac32,2,\tfrac52 &
\text{\thedimension} & 5 & 3 \\
\hline
\end{array}
\]
\

\[
\begin{array}{|c|l|}
\hline
 & \text{Recurrences} \\
\hline
F &
\begin{aligned}
(n+&3)^5\,x_{n+3}
- 2(2n+5)\bigl(14n^4+140n^3+532n^2+910n+591\bigr)\,x_{n+2} \\
&+ 4(n+2)^3\bigl(196n^2+784n+843\bigr)\,x_{n+1}
- 1152 (n+1)^2 (2n+3) (n+2)^2 \,x_{n}=0
\end{aligned}
\\
\hline
A &
\begin{aligned}
(n+&3)^6\,A_{n+3}
 -4(2n+5)^2 (14(n+2)^4+28(n+2)^3+28(n+2)^2+14(n+2)+3)\, A_{n+2} \\
 &+16(n+2)^2(2n+3)(2n+5)(196(n+1)^2+392(n+1)+255)\, A_{n+1} \\  &-9216(n+1)(n+2)(2n+1)(2n+3)^2(2n+5)\,A_{n}=0 \\
\end{aligned}
 \\
\hline
\end{array}
\]

\[
\begin{array}{|c|l|}
\hline
 & \text{The $L_{\gt}=zL$ operator in CPF (with ${R_k(\gt)}$)} \\
\hline
F &
\begin{aligned}
z^3 \cdot & [1152(\theta+1)^2(\theta+2)^2(2\theta+3)] \\
 &-z^2\cdot [4(\theta+1)^3(196\,\theta^2+392\,\theta+255)] \\
 & +z \cdot [2(2\theta+1)(14\,\theta^4+28\,\theta^3+28\,\theta^2+14\,\theta+3)]
  -\theta^5
\end{aligned}
\\
\hline
A &
\begin{aligned}
z^3 \cdot & [9216(\theta+1)(\theta+2)(2\theta+1)(2\theta+3)^2(2\theta+5)] \\
 & -z^2 \cdot [16(\theta+1)^2(2\theta+1)(2\theta+3)(196\,\theta^2+392\,\theta+255) \\
 & + z \cdot  [4(2\theta+1)^2(14\,\theta^4+28\,\theta^3+28\,\theta^2+14\,\theta+3)
  -\theta^6
\end{aligned}
 \\
\hline
\end{array}
\]

\[
\begin{array}{|c|l|}
\hline
 & \text{The $L_{\gt} = z L$ operator (with $\widetilde{R_k}(z)$)} \\
\hline
F &
\begin{aligned}
(4z-1)&(16z-1)(36z-1)\,\theta^5
 + 20z(864z^2-196z+7)\,\theta^4 \\
 &+12z(4224z^2-673z+14)\,\theta^3
  +4z(18144z^2-2137z+28)\,\theta^2 \\
 &+4z(12672z^2-1157z+10)\,\theta
  +6z(2304z^2-170z+1)
\end{aligned}
\\
\hline
A &
\begin{aligned}
(16z-1)&(64z-1)(144z-1)\,\theta^6
 +96z(13824z^2-784z+7)\,\theta^5 \\
 &+8z(599040z^2-23600z+119)\,\theta^4
  +16z(552960z^2-15840z+49)\,\theta^3 \\
 &+16z(546624z^2-11941z+24)\,\theta^2
  +8z(542592z^2-9492z+13)\,\theta \\
 &+12z(69120z^2-1020z+1)
\end{aligned}
\\
\hline
\end{array}
\]

\[
\begin{array}{|c|l|}
\hline
 & \text{The $L$ operator in standard CPF} \\
\hline
F &
\begin{aligned}
 z^4 & (4z - 1)(16z - 1)(36z - 1)\, D^5 +
 10z^3(4032z^3 - 1176z^2 + 70z - 1)\, D^4 \\
 & + z^2(211968z^3 - 51196z^2 + 2408z - 25)\, D^3
  + 3z(126720z^3 - 23992z^2 + 812z - 5)\, D^2 \\
 & +(193536z^3 - 25956z^2 + 516z - 1)\, D
  + 6(2304z^2 - 170z + 1)
\end{aligned}
\\
\hline
A &
\begin{aligned}
z^5&(16z-1)(64z-1)(144z-1) D^6 
+3z^4\left(1179648z^3-87808z^2+1344z-5\right) D^5 \\
&+z^3\left(27648000z^3-1756800z^2+22232z-65\right) D^4 \\
&+2z^2\left(42024960z^3-2198400z^2+21728z-45\right) D^3 \\
&+z\left(93312000z^3-3790800z^2+26424z-31\right) D^2 \\
&+\left(28200960z^3-797040z^2+3120z-1\right)D 
+12\left(69120z^2-1020z+1\right)
\end{aligned}
 \\
\hline
\end{array}
\]


\subsection{The case $d=7$} 

\vspace{.2cm}
\[
\begin{array}{|c|c|c|c|c|c|c|}
\hline
 & \text{Finite singularities} & \text{Exponents at each finite point} & \text{Exponents at } \infty & q & m & r \\
\hline
F & 1,1/9, 1/25,1/49 & 0,1,2,\tfrac52,3,4,5 & 1,1,2,2,3,3 &
\text{\thedimension} & 5 & 4 \\
\hline \stepcounter{dimension}
A & 1/4, 1/36, 1/100, 1/196 & 0,1,2,\tfrac52,3,4,5 & \tfrac12,1,\tfrac32,2,\tfrac52,3,\tfrac72 &
\text{\thedimension} & 6 & 4 \\
\hline
\end{array}
\]
\

\[
\begin{array}{|c|l|}
\hline
 & \text{Recurrences} \\
\hline
F &
\begin{aligned}
(n+&4)^6\, x_{n+4} \\
&- \left( 84\,{n}^{6}+1764\,{n}^{5}+15498\,{n}^{4}+72912\,{n}^{3}+193716\,{n}^{2}+275562\,n+163951 \right) \, x_{n+3} \\
&+3\,(n+3)^{2} \left( 658\,{n}^{4}+7896\,{n}^{3}+35928\,{n}^{2}+73440\,n+56879 \right)\, x_{n+2} \\
&-2\,(n+2)^{2} (n+3)^{2} \left( 6458\,{n}^{2}+32290\,n+41337 \right)\, x_{n+1} \\
&+11025\, (n+1)^{2} \, (n+2)^{2} \, (n+3)^{2}\, x_{n} = 0.
\end{aligned}
\\
\hline
A &
\begin{aligned}
(n+&4)^7 A_{n+4} \\
  &-2(2n+7) \bigl(84n^6+1764n^5+15498n^4+72912n^3+193716n^2+275562n+163951 \bigr) A_{n+3} \\
  &+12(n+3)(2n+5)(2n+7) \bigl(658n^4+7896n^3+35928n^2+73440n+56879\bigr) A_{n+2} \\
  &-16(n+2)(n+3)(2n+3)(2n+5)(2n+7)     \bigl(6458n^2+32290n+41337\bigr) A_{n+1} \\
  &+176400 (n+1)(n+2)(n+3)(2n+1)(2n+3)(2n+5)(2n+7)A_{n}=0.
\end{aligned}
 \\
\hline
\end{array}
\]

\[
\begin{array}{|c|l|}
\hline
 & \text{The $L_{\gt}=zL$ operator in CPF (with ${R_k(\gt)}$)} \\
\hline
F &
\begin{aligned}
z^4 \cdot &[11025 (\theta+1)^2(\theta+2)^2(\theta+3)^2] \\
 &-z^3 \cdot [2(\theta+1)^2(\theta+2)^2(6458\,\theta^2+19374\,\theta+15505)] \\
 &+z^2 \cdot [3(\theta+1)^2(658\,\theta^4
    +2632\,\theta^3+4344\,\theta^2+3424\,\theta+1071)] \\
 &-z \cdot [84\,\theta^6+252\,\theta^5+378\,\theta^4
    +336\,\theta^3+180\,\theta^2+54\,\theta+7] + \theta^6
\end{aligned}
\\
\hline
A &
\begin{aligned}
z^4 \cdot &[176400(\theta+1)(\theta+2)(\theta+3)(2\theta+1)(2\theta+3)(2\theta+5)(2\theta+7)] \\
 &-z^3 \cdot [16(\theta+1)(\theta+2)(2\theta+1)(2\theta+3)(2\theta+5)(6458\,\theta^2+19374\,\theta+15505)] \\
 &+z^2 \cdot [12(\theta+1)(2\theta+1)(2\theta+3)(658\,\theta^4+2632\,\theta^3+4344\,\theta^2+3424\,\theta+1071)] \\
 &-z \cdot [2(2\theta+1)(84\,\theta^6+252\,\theta^5+378\,\theta^4
        +336\,\theta^3+180\,\theta^2+54\,\theta+7)] + \theta^7
\end{aligned}
 \\
\hline
\end{array}
\]

\[
\begin{array}{|c|l|}
\hline
 & \text{The $L_{\gt} = z L$ operator (with $\widetilde{R_k}(z)$)} \\
\hline
F &
\begin{aligned}
(z&-1)(9z-1)(25z-1)(49z-1)\,\theta^6 \\
 &+36z(3675z^3-3229z^2+329z-7)\,\theta^5 \\
 &+18z(35525z^3-23967z^2+1711z-21)\,\theta^4 \\
 &+24z(66150z^3-35199z^2+1843z-14)\,\theta^3 \\
 &+3z(709275z^3-306590z^2+12263z-60)\,\theta^2 \\
 &+6z(242550z^3-87852z^2+2783z-9)\,\theta \\
 &+7z(56700z^3-17720z^2+459z-1)
\end{aligned}
\\
\hline
\end{array}
\]

\[
\begin{array}{|c|l|}
\hline
 & \text{The $L_{\gt} = z L$ operator (with $\widetilde{R_k}(z)$)} \\
\hline
A &
\begin{aligned}
(4z&-1)(36z-1)(100z-1)(196z-1)\,\theta^7 \\
&+168z(235200z^3-51664z^2+1316z-7)\,\theta^6 \\
&+24z(9466800z^3-1589560z^2+28099z-84)\,\theta^5 \\
&+60z(11524800z^3-1513848z^2+19350z-35)\,\theta^4 \\
&+48z(24876075z^3-2626582z^2+25263z-29)\,\theta^3 \\
&+48z(24130050z^3-2112201z^2+15887z-12)\,\theta^2 \\
&+4z(144074700z^3-10821620z^2+66159z-34)\,\theta \\
&+14z(7938000z^3-531600z^2+2754z-1)
\end{aligned}
\\
\hline
\end{array}
\]

\[
\begin{array}{|c|l|}
\hline
 & \text{The $L$ operator in standard CPF} \\
\hline
F &
\begin{aligned}
z^5&(z-1)(9z-1)(25z-1)(49z-1) D^6 \\
&+3z^4\left(99225z^4-103328z^3+13818z^2-504z+5\right) D^5 \\
&+z^3\left(2679075z^4-2433386z^3+277548z^2-8358z+65\right) D^4 \\
&+6z^2\left(1620675z^4-1250292z^3+117130z^2-2744z+15\right) D^3 \\
&+z\left(13693050z^4-8617996z^3+623925z^2-10218z+31\right) D^2 \\
&+\left(5953500z^4-2852224z^3+142335z^2-1284z+1\right) D \\
&+7\left(56700z^3-17720z^2+459z-1\right).
\end{aligned}
\\
\hline
A &
\begin{aligned}
z^6&(4z-1)(36z-1)(100z-1)(196z-1) D^7 \\
 &+21 z^5 (4704000 z^4 - 1239936 z^3 + 42112 z^2 - 392 z + 1) D^6 \\
 &+4 z^4 (303760800 z^4 - 71017520 z^3 + 2103114 z^2 - 16674 z + 35) D^5 \\
 &+10 z^3 (651974400 z^4 - 132581456 z^3 + 3332988 z^2 - 21630 z + 35) D^4 \\
 &+z^2 (15428826000 z^4 - 2654770720 z^3 + 54442728 z^2 - 271368 z + 301) D^3 \\
 &+3 z (4797198000 z^4 - 669604880 z^3 + 10495536 z^2 - 35772 z + 21) D^2 \\
 &+(3889620000 z^4 - 409234560 z^3 + 4327884 z^2 - 7732 z + 1) D \\
 &+14 (7938000 z^3 - 531600 z^2 + 2754 z - 1)
\end{aligned}
 \\
\hline
\end{array}
\]


\hfill
\subsection{The case $d=8$} 

\vspace{.2cm}
\[
\begin{array}{|c|c|c|c|c|c|c|}
\hline
 & \text{Finite singularities} & \text{Exponents at each finite point} & \text{Exponents at } \infty & q & m & r \\
\hline
F & 1/4, 1/16, 1/36, 1/64 & 0, 1, 2, \tfrac52, 3, 4, 5
   & 1, 1, 2, 2, 2, 3, 3 &
\text{\thedimension} & 6 & 4 \\
\hline \stepcounter{dimension}
A & 1/16, 1/64, 1/144, 1/256 & 0, 1, 2, 3, 3, 4, 5, 6
   &  \tfrac12,1,\tfrac32,2,2,\tfrac52,3,\tfrac72  &
\text{\thedimension} & 7 & 4 \\
\hline
\end{array}
\]
\hfill

\[
\begin{array}{|c|l|}
\hline
 & \text{Recurrences} \\
\hline
F &
\begin{aligned}
(n+&4)^7\, x_{n+4} \\
&-2\, (2n+7)\left( 30\,n^6 + 630\, n^5 + 5544\, n^4 + 26166\, n^3 + 69849\, n^2 + 99981 \, n + 59944 \right) \, x_{n+3} \\
&+12\, (n+3)^3\left( 364 \, n^4 + 4368\, n^3 + 20021\, n^2 + 41502\, n + 32816 \right) \, x_{n+2} \\
&-128\, (n+2)^2(n+3)^2(2n+5) \left( 205\, n^2 + 1025\, n + 1374\right)\, x_{n+1} \\
&+147456\, (n+1)^2 \, (n+2)^3 \, (n+3)^2\, x_{n} = 0.
\end{aligned}
\\
\hline
A &
\begin{aligned}
(n+&4)^8 \,A_{n+4} \\
&-4(2n+7)^2
\left(30n^6+630n^5+5544n^4+26166n^3
+69849n^2+99981n+59944\right)\, A_{n+3} \\
&+48(n+3)^2(2n+5)(2n+7)
\left(364n^4+4368n^3+20021n^2+41502n+32816\right)\, A_{n+2} \\
&\quad -1024(n+2)(n+3)(2n+3)(2n+5)^2(2n+7)
\left(205n^2+1025n+1374\right)\, A_{n+1} \\
&+2359296(n+1)(n+2)^2(n+3)(2n+1)(2n+3)(2n+5)(2n+7)\, A_n =0
\end{aligned}
\\
\hline
\end{array}
\]

\[
\begin{array}{|c|l|}
\hline
 & \text{The $L_{\gt}=zL$ operator in CPF (with ${R_k(\gt)}$)} \\
\hline
F &
\begin{aligned}
z^4 \cdot &[147456(\theta+1)^2(\theta+2)^3(\theta+3)^2] \\
 &-z^3 \cdot [128(\theta+1)^2(\theta+2)^2(2\theta+3)
   (205\,\theta^2+615\,\theta+554)] \\
 &+z^2 \cdot [12(\theta+1)^3
   (364\,\theta^4+1456\,\theta^3+2549\,\theta^2+2186\,\theta+776)] \\
 &-z \cdot [2(2\theta+1)(30\,\theta^6+90\,\theta^5+144\,\theta^4
           +138\,\theta^3+81\,\theta^2+27\,\theta+4)] \\
 &+ \theta^7
\end{aligned}
 \\
\hline
A &
\begin{aligned}
z^4 \cdot &[2359296(\theta+1)(\theta+2)^2(\theta+3)
  (2\theta+1)(2\theta+3)(2\theta+5)(2\theta+7)] \\
 &-z^3 \cdot [1024(\theta+1)(\theta+2)(2\theta+1)
  (2\theta+3)^2(2\theta+5) (205\,\theta^2+615\,\theta+554)] \\
 &+z^2 \cdot [48z(\theta+1)^2(2\theta+1)(2\theta+3)
  (364\,\theta^4+1456\,\theta^3+2549\,\theta^2+2186\,\theta+776)] \\
 &-z \cdot [4(2\theta+1)^2 (30\,\theta^6+90\,\theta^5+144\,\theta^4
  +138\,\theta^3+81\,\theta^2+27\,\theta+4)] \\
 &+ \theta^8
\end{aligned}
 \\
\hline
\end{array}
\]

\[
\begin{array}{|c|l|}
\hline
 & \text{The $L_{\gt} = z L$ operator (with $\widetilde{R_k}(z)$)} \\
\hline
F &
\begin{aligned}
(4z-&1)(16z-1)(36z-1)(64z-1)\,\theta^7\\
&+84z(24576z^3-6560z^2+364z-5)\,\theta^6\\
&+12z(1007616z^3-206432z^2+8009z-63)\,\theta^5\\
&+60z(638976z^3-103008z^2+2913z-14)\,\theta^4\\
&+12z(5910528z^3-769504z^2+16437z-50)\,\theta^3\\
&+6z(12730368z^3-1374336z^2+22870z-45)\,\theta^2\\
&+2z(22118400z^3-2032384z^2+27084z-35)\,\theta\\
&+8z(1327104z^3-106368z^2+1164z-1)
\end{aligned}
\\
\hline
\end{array}
\]

\[
\begin{array}{|c|l|}
\hline
 & \text{The $L_{\gt} = z L$ operator (with $\widetilde{R_k}(z)$)} \\
\hline
A &
\begin{aligned}
(16z-&1)(64z-1)(144z-1)(256z-1)\,\theta^8\\
&+384z(1572864z^3-104960z^2+1456z-5)\,\theta^7\\
&+24z(170655744z^3-8705024z^2+83728z-161)\,\theta^6\\
&+24z(638582784z^3-25445376z^2+176224z-203)\,\theta^5\\
&+48z(718061568z^3-22891200z^2+118183z-85)\,\theta^4\\
&+24z(1976303616z^3-51713280z^2+205496z-95)\,\theta^3\\
&+4z(9672523776z^3-213468416z^2+673188z-205)\,\theta^2\\
&+4z(4225499136z^3-80968704z^2+209064z-43)\,\theta\\
&+16z(185794560z^3-3191040z^2+6984z-1) \\ &
\end{aligned}
\\
\hline
\end{array}
\]

\[
\begin{array}{|c|l|}
\hline
 & \text{The $L$ operator in standard CPF} \\
\hline
F &
\begin{aligned}
z^6&(4z - 1)(16z - 1)(36z - 1)(64z - 1) \, D^7 \\
 &+21 z^5 (245760 z^4 - 78720 z^3 + 5824 z^2 - 140 z + 1) \, D^6(z) \\
 &+4 z^4 (15925248 z^4 - 4522496 z^3 + 291567 z^2 - 5964 z + 35) \, D^5 \\
 &+10 z^3 (34504704 z^4 - 8513792 z^3 + 465210 z^2 - 7770 z + 35) \, D^4  \\
 &+z^2 (833421312 z^4 - 173636608 z^3 + 7715232 z^2 - 98460 z + 301) \, D^3 \\
 &+ 3 z (270729216 z^4 - 45585920 z^3 + 1539024 z^2 - 13290 z + 21) \, D^2 \\
 &+(244187136 z^4 - 30806016 z^3 + 694464 z^2 - 3076 z + 1) \, D \\
 &+ 8 (1327104 z^3 - 106368 z^2 + 1164 z - 1) \\ &
\end{aligned}
\\
\hline
A &
\begin{aligned}
z^7&(16z-1)(64z-1)(144z-1)(256z-1)\,D^8 \\
&+4z^6\left(415236096z^4-33587200z^3+628992z^2-3840z+7\right)\,D^7 \\
&+2z^5\left(13410238464z^4-974368768z^3+16170432z^2-85932z+133\right)\,D^6 \\
&+6z^4\left(33492566016z^4-2152300544z^3+31004736z^2-139272z+175\right)\,D^5 \\
&+3z^3\left(243184435200z^4-13535104000z^3+164382704z^2-597480z+567\right)\,D^4 \\
&+6z^2\left(204043714560z^4-9546680320z^3+93559600z^2-256320z+161\right)\,D^3 \\
&+z\left(821955133440z^4-30871296000z^3+227031552z^2-411004z+127\right)\,D^2 \\
&+\left(157553786880z^4-4380917760z^3+21001536z^2-18488z+1\right)\,D \\
&+16\left(185794560z^3-3191040z^2+6984z-1\right) \\ &
\end{aligned}
\\
\hline
\end{array}
\]

\newpage
\addcontentsline{toc}{section}{References}

\def\cprime{$'$}

\end{document}